\documentclass[11pt]{amsart}

\usepackage[foot]{amsaddr}
\usepackage[toc]{appendix} 
\usepackage{latexsym}
\usepackage{amsfonts}
\usepackage{amssymb}
\usepackage{amsmath}
\usepackage{mathrsfs}
\usepackage{bbm}
\usepackage{dsfont}
\usepackage{hyperref}
\allowdisplaybreaks

\usepackage[toc]{appendix}

\newcommand{\be}{\begin{equation}}
\newcommand{\ee}{\end{equation}}
\newcommand{\bea}{\begin{eqnarray}}
\newcommand{\eea}{\end{eqnarray}}
\newcommand{\bean}{\begin{eqnarray*}}
\newcommand{\eean}{\end{eqnarray*}}
\newcommand{\brray}{\begin{array}}
\newcommand{\erray}{\end{array}}

\newcommand{\bdfn}{\begin{dfn}\rm}
\newcommand{\bthm}{\begin{thm}}
\newcommand{\blmma}{\begin{lmma}}
\newcommand{\bppsn}{\begin{ppsn}}
\newcommand{\bcrlre}{\begin{crlre}}
\newcommand{\bxmpl}{\begin{xmpl}}
\newcommand{\brmrk}{\begin{rmrk}\rm}

\newcommand{\edfn}{\end{dfn}}
\newcommand{\ethm}{\end{thm}}
\newcommand{\elmma}{\end{lmma}}
\newcommand{\eppsn}{\end{ppsn}}
\newcommand{\ecrlre}{\end{crlre}}
\newcommand{\exmpl}{\end{xmpl}}
\newcommand{\ermrk}{\end{rmrk}}

\newcommand{\bbc}{\mathbb{C}}
\newcommand{\bbz}{\mathbb{Z}}

\newcommand{\bbn}{\mathbb{N}}
\newcommand{\bbr}{\mathbb{R}}

\newcommand{\cla}{\mathcal{A}}

\newcommand{\clh}{\mathcal{H}}

\newcommand{\clk}{\mathcal{K}}

\makeatletter
\let\@wraptoccontribs\wraptoccontribs
\makeatother

\title{On  Multiparameter CAR Flows}
\author{C. H. Namitha  and S. Sundar}
\address{The Institute of Mathematical Sciences, A CI of Homi Bhabha National Institute, 4th cross street, CIT Campus, Taramani, Chennai, India, 600113}
\email{namithachanguli7@gmail.com, sundarsobers@gmail.com}

\usepackage{enumerate}
\usepackage{bbm}
\usepackage{cleveref}

\newtheorem{definition}{Definition}[section]
\newtheorem{prop}[definition]{Proposition}
\newtheorem{theorem}[definition]{Theorem}

\newtheorem{lemma}[definition]{Lemma}
\newtheorem{remark}[definition]{Remark}
\numberwithin{equation}{section}
\newcommand{\RNum}[1]{\uppercase\expandafter{\romannumeral #1\relax}}

\begin{document}
\maketitle

\begin{abstract}
Let $P$ be a pointed, closed convex cone in $\bbr^d$. We prove that for two pure isometric representations $V^{(1)}$ and $V^{(2)}$ of $P$, the associated CAR flows $\beta^{V^{(1)}}$ and $\beta^{V^{(2)}}$ are cocycle conjugate if and only if $V^{(1)}$ and $V^{(2)}$ are unitarily equivalent.  We also give a complete description of pure isometric representations of $P$ with commuting range projections that give rise to type I CAR flows.  We show that such an isometric representation  is completely reducible with each irreducible component being a pullback of the shift semigroup $\{S_t\}_{t \geq 0}$ on $L^2[0,\infty)$. 
We also compute the index and   the gauge group of  the associated CAR flows and show that the action of the gauge group on the set of normalised units need not be transitive. 
 \end{abstract}

\noindent {\bf AMS Classification No. :} {Primary 46L55; Secondary 46L99.}  


\section{Introduction}
Broadly speaking, the subject of  irreversible non-commutative dynamics is concerned with the study of 
dynamical systems where instead of a group action on  a non-commutative space (Hilbert spaces, $C^{*}$-algebras, von Neumann algebras), we have
 a semigroup action. The  operator algebraic aspects of such irreversible dynamical systems  have received the attention
 of many authors over the years. Some of the topics that were investigated in detail and continue to be a source for  much research are 
 semigroup $C^{*}$-algebras (\cite{Li_Oberwolfach}), semi-crossed products of non self-adjoint algebras (\cite{KDFK}), dilation theory of semigroups of contractions and CP-semigroups (\cite{Shalit_dilation}), $E_0$-semigroups and 
 product systems (\cite{arveson}). This paper comes under the topic of $E_0$-semigroups and product systems. 
 
 An $E_0$-semigroup (CP-semigroup)  over a semigroup $P$ on $B(\mathcal{H})$ is a semigroup $\alpha:=\{\alpha_x\}_{x \in P}$ of unital, normal $^*$-endomorphisms (CP-maps) of $B(\mathcal{H})$. If $P$ has a topology, we require the semigroup $\alpha$ to satisfy an appropriate continuity hypothesis. 
The study of such semigroups, when $P=[0,\infty)$, has a long history which dates back to  Powers' works (\cite{Powers_Index}, \cite{Powers_TypeIII}).  Arveson wrote several influential papers on the subject (\cite{Arv_Fock}, \cite{Arv_Fock2}, \cite{Arv_Fock3}, \cite{Arv_Fock4}) and also authored  \cite{arveson} which is  the standard reference for the subject. More on Arveson's contribution and by others to the subject of $E_0$-semigroups can be found in a  survey article (\cite{Izumi}) by Izumi. 

 In the last fifteen years, several papers (\cite{Shalit},\cite{Shalit_2008}, \cite{Solel}, \cite{Shalit_Solel}, \cite{Skeide_Shalit}, \cite{MuruganSP}, \cite{anbuCAR}, \cite{Anbu_Vasanth}, \cite{Anbu_Sundar}, \cite{SUNDAR}) appeared where  semigroups of endomorphisms/CP-maps and product systems, over more general monoids, were considered. Moreover, it was demonstrated  that  significant differences show up in the multiparameter case. A few features that are in stark contrast to the $1$-parameter case are listed below. 
\begin{enumerate}
\item A CP-semigroup over $\mathbb{N}^{3}$ need not have a dilation to an $E_0$-semigroup (\cite{Shalit_Skeide_2011}).
\item CCR and CAR flows need not be cocycle conjugate in the multiparameter case (\cite{srinivasanccr}).
\item In the multiparameter case, decomposable product systems need not be spatial (\cite{SUNDAR},  \cite{Namitha_Sundar}). 
\end{enumerate}
These contrasting phenomena make the multiparameter theory interesting, and the authors  believe that multiparameter $E_0$-semigroups are objects worthy of investigation. 
Here, we study a class of $E_0$-semigroups called CAR flows and we classify a subclass of them.

In view of the bijective correspondence between the class of product systems and the class of $E_0$-semigroups (\cite{Arv}, \cite{Skeide}, \cite{MuruganSP}), we explain the problem studied and the results obtained in the language of product systems. A product system is a measurable field of separable Hilbert spaces $E:=\{E(x)\}_{x \in P}$ endowed with a multiplication that is compatible with the measurable structure. 

The two simplest product systems, whose definitions we recall,  are the ones associated with CCR and CAR flows.  Let $\mathcal{H}$ be a separable Hilbert space.  (All  Hilbert spaces considered in this paper are tacitly assumed to be separable.) We denote the symmetric Fock space of $\mathcal{H}$ by $\Gamma_s(\mathcal{H})$ and the antisymmetric Fock space of $\mathcal{H}$ by $\Gamma_a(\mathcal{H})$. Throughout this paper, the letter $P$ stands for a closed, convex cone in $\bbr^d$ that is pointed, i.e. $P \cap -P=\{0\}$ and spanning, i.e. $P-P=\bbr^d$. Let $V:=\{V_x\}_{x \in P}$ be a strongly continuous semigroup of isometries on $\mathcal{H}$. We call such a semigroup of isometries an isometric representation of $P$ on $\mathcal{H}$. We assume $V$ is pure, i.e. $\displaystyle \bigcap_{x \in P}V_x\clh=\{0\}$. 

 Consider the field of Hilbert spaces $F^V:=\{F^V(x)\}_{x \in P}$,  $F^V(x):=\Gamma_s(Ker(V_x^*))$ for $x \in P$. We impose a measurable structure on $F^V$ as follows. For every $x \in P$, we can view $F^V(x)$ as a subspace of $\Gamma_s(\clh)$ as the embedding $Ker(V_x^*) \subset \clh$ induces an embedding of $\Gamma_s(Ker(V_x^*))$ in $\Gamma_s(\clh)$. Let $\Gamma$ be the set of all maps $t:P \to \Gamma_s(\clh)$ such that 
 \begin{enumerate}
 \item[(a)] the map $t$ is weakly measurable, and
 \item[(b)] for $x \in P$, $t(x) \in F^V(x)$.
 \end{enumerate}. Then, $F^V$ is a measurable field of Hilbert spaces with $\Gamma$ being the space of measurable sections. Define a product rule on $F^V$  by
 \begin{equation}
 \label{multiplication formula}
 e(\xi)e(\eta)=e(\xi+V_x\eta)\end{equation}
 for $\xi \in Ker(V_x^*)$ and $\eta \in Ker(V_y^*)$. Here, $\{e(\xi):\xi \in Ker(V_x^*)\}$ denotes the collection of exponential vectors in the symmetric Fock space $\Gamma_s(Ker(V_x^*))
$. Then, the multiplication is compatible with  the measurable structure  making $F^V$ a product system. The product system $F^V$ is called the product system of the  CCR flow associated with $V$. The corresponding $E_0$-semigroup $\alpha^V$ is called the CCR flow associated with $V$. Concrete examples of multiparameter CCR flows and their intrinsic properties like index, gauge group, type  were analysed in \cite{Anbu_Sundar}, \cite{Anbu_Vasanth} and in \cite{Piyasa_Sundar1}.

 For $x \in P$, let $\Omega_x$ be the vacuum vector of $\Gamma_a(Ker(V_x^{*}))$. 
 Consider the field of Hilbert spaces $E^{V}:=\{E^{V}(x)\}_{x\in P}$, where $E^{V}(x):=\Gamma_{a}(Ker(V_{x}^{*}))$ for $x\in P$.  The measurable structure on $E^V$ is defined as in the CCR case.  Define a multiplication on $E^{V}$ as follows:
\begin{align}\label{multiplication} \xi \cdot \eta=V_{x}\eta_1 \wedge V_x \eta_2 \cdots \wedge V_x \eta_n \wedge \xi_1 \wedge \xi_2 \wedge \cdots \wedge \xi_m
 \end{align} for $\xi=\xi_1 \wedge \xi_2 \cdots \wedge \xi_m\in \Gamma_a(Ker(V_{x}^{*})$,  $\eta=\eta_1 \wedge \eta_2 \wedge \cdots \wedge \eta_n \in \Gamma_a(Ker(V_{y}^{*}))$ and $x,y\in P$.
For $m=0$ or $n=0$, Eq. \ref{multiplication} is interpreted as follows:
\begin{align*}
 \Omega_x \cdot \Omega_y&=\Omega_{x+y},\\
 \Omega_x \cdot \eta&=V_{x}\eta_1 \wedge V_x \eta_2 \wedge \cdots \wedge V_x \eta_n, \\
 \xi \cdot \Omega_y&=\xi_1 \wedge \xi_2 \wedge \cdots \wedge \xi_m.
\end{align*}
Then, $E^V$ is a product system and is called the product system of the CAR flow associated with $V$. The corresponding $E_0$-semigroup $\beta^V$ is called the CAR flow associated with $V$.

It is known that in the $1$-parameter case, i.e. when $P=[0,\infty)$, $E^V$ and  $F^V$ are isomorphic (\cite{Powers_Rob}). Then, it follows from the work of Arveson (\cite{Arv_Fock}) that $1$-parameter CAR flows are classified by a single numerical invariant called  index that takes values in $\{0, 1,2,\cdots\} \cup \{\infty\}$. In particular, the map 
\[
V \to E^V\]
is injective. Also, $1$-parameter CAR flows are  type I.  Here, we take up the multiparameter case, and we completely classify type I CAR flows associated with isometric representations  with commuting range projections.

Before we state our results, we mention here that, in the multiparameter case, 
the study of CCR flows and CAR flows are not the same thing. For, it was demonstrated in \cite{srinivasanccr} that $E^V$ and $F^V$ need not be isomorphic. In \cite{anbuCAR}, Arjunan studied the decomposability of the product system $E^V$ when $V$ is the shift semigroup associated with a free and transitive action of $P$. He showed that for a large class of isometric representations $V$, $E^V$ fails to be decomposable, and hence not isomorphic to the product system of a CCR flow. 

Our first result concerning CAR flows is given below. 
\begin{theorem}
\label{oneone}
    Let $V^{(1)}$ and $V^{(2)}$ be pure isometric representations of $P$ on Hilbert spaces $\clh_1$ and $\clh_2$ respectively. The product systems $E^{V^{(1)}}$ and $E^{V^{(2)}}$ are isomorphic  if and only if $V^{(1)}$ and $V^{(2)}$ are unitarily equivalent, i.e. there exists a unitary $U:\clh_1 \to \clh_2 $ such that for every $a \in P$,
    \[
    UV_{a}^{(1)}U^{*}=V_{a}^{(2)}.
    \]
\end{theorem}
 The above theorem was stated as Prop. 4.7 in \cite{srinivasanccr}. However, the proof given there is incorrect. The author argues that in view of Prop. 4.1 of \cite{srinivasanccr}, it suffices to prove that the gauge group of a CAR flow acts transitively on the set of normalised units, and he gives an incorrect proof of this assertion. In fact, the last assertion is  false.  We show by counterexamples that the gauge group of a CAR flow need not act transitively on the set of normalised units. 

 Thm. \ref{oneone} says that the task of parametrising/listing  CAR flows is equivalent to the problem of parametrising isometric representations of $P$. However, the process of inducing isometric representations (\cite{Piyasa_Sundar1}, \cite{Dahya}) of  $\bbn^2$ to that of $\bbr_+^2$ allows us to conclude  that  the classification problem of isometric representations of $\bbr_{+}^{2}$ is at least as hard as describing   the dual of $C^{*}(\bbz_2*\bbz)$ which is known to be pathological, i.e. it is not a standard Borel space. Thus, describing a good parameterisation of  all CAR flows is beyond the scope of this paper and the authors. 
 Neverthless, we show that a suitable subclass, i.e. the class of type I CAR flows associated with isometric representations with commuting range projections can be completely classified and described in concrete terms. 
 
 Recall that an isometric representation $V=\{V_a\}_{a \in P}$ is said to have \textbf{commuting range projections} if $\{V_aV_{a}^{*}:a \in P\}$ is a commuting family of projections, and recall that a product system $E$ is said to be type I if  the only  subsystem of $E$ that contains all the units of $E$ is $E$. A unit of $E$ is a non-zero multiplicative  section of $E$. 
 
 We fix notation to describe our results.  Let $\mathbb{N}_\infty:=\{1,2,\cdots\} \cup \{\infty\}$. Let $P^{*}$ be the dual cone of $P$, i.e. 
 \[
 P^{*}:=\{\lambda \in \bbr^d: \langle \lambda|x \rangle \geq 0 \textrm{~for $x \in P$} \}.
 \]
 We denote by $S(P^*)$ the unit sphere of $P^*$, i.e.
 \[
 S(P^*):=\{\lambda \in P^*: \langle \lambda| \lambda \rangle=1\}.
 \]

 Let $\lambda \in S(P^*)$, let $k \in \bbn_\infty$, and let $\clk$ be a  Hilbert space of dimension $k$. Denote the one parameter shift semigroup on $\clh:=L^{2}([0,\infty),\clk)$ by $S^{(k)}:=\{S_t^{(k)}\}_{t \geq 0}$. Recall that for $t \geq 0$, $S_t^{(k)}$ is the isometry on $\clh$ defined by 
 \begin{equation*}
S_t^{(k)}f(x):=\begin{cases}
 f(x-t) & \mbox{ if
} x-t \geq 0,\cr
   &\cr
   0 &  \mbox{ otherwise}.
         \end{cases}
\end{equation*}
For $a \in P$, let $S_a^{(\lambda,k)}:=S_{\langle \lambda|a \rangle}^{(k)}$. Then, $S^{(\lambda,k)}:=\{S_{a}^{(\lambda,k)}\}_{a \in P}$ is an isometric representation of $P$ on $\clh$. If $k=1$, we denote $S^{(\lambda,k)}$ by $S^\lambda$. 

For a non-empty countable subset $I$, an injective map $\lambda:I \to S(P^*)$ and a map $k:I \to \bbn_\infty$, set 
\[
S^{(\lambda,k)}:=\bigoplus_{i \in I}S^{(\lambda_i,k_{i})}.
\]
 Let $E^{(\lambda,k)}$ be the product system of  the CAR flow associated with $S^{(\lambda,k)}$. 

With the above notation, we have the following classification result which is the main result of this paper.   
\begin{theorem}
\label{main intro}
For every non-empty countable set $I$, an injective map $\lambda:I \to S(P^*)$ and a map $k:I \to \bbn_\infty$, the product system $E^{(\lambda,k)}$ is type I.  Conversely, suppose $V$ is an isometric representation of $P$ with commuting range projections such that $E^V$ is type $I$.  Then,  $E^V$ is  isomorphic to $E^{(\lambda,k)}$ (equivalently, $V$ is unitarily equivalent to $S^{(\lambda,k)}$) for a  non-empty countable set $I$, an injective map $\lambda:I \to S(P^*)$ and a  map $k:I \to \bbn_\infty$. Moroever, the maps $\lambda$ and $k$ are unique up to conjugacy.  
\end{theorem}

We explain briefly the ideas involved in the the proof of the converse part. A certain `universal irreversible' dynamical system that encodes all pure semigroups of isometries with commuting range projections was constructed in \cite{Sundar_Ore} (see also \cite{Sundar_NYJM}). The irreversible system  is given by the pair $(X_u,P)$ where
\[
X_u:=\{A \subset \bbr^d: A \neq \emptyset, A \neq \bbr^d, -P+A \subset A, 0 \in A, \textrm{~$A$ is closed}\}.
\]
The topology that we impose on $X_u$ is the Fell topology. The semigroup $P$ acts on $X_u$ by translations.

The results of \cite{sundarkms}  allow us to conclude that if we focus on one unit 
$u:=\{u_a\}_{a \in P}$ at a time, then we can assume that $V$ is the shift semigroup $V
^\mu$ on $X_u$ associated with a translation invariant Radon measure $\mu$ on $X_u$. Making use of the equality 
\[
u_au_b=u_bu_a
\]
and by  few  computations, we show that the support of $\mu$ can be identified with $[0,\infty)$. Under this identification, the action of $P$ on $[0,\infty)$ is then given by 
\[
[0,\infty) \times P \ni (x,a) \to x+\langle \lambda|a \rangle \in [0,\infty) \]
for a unique $\lambda \in  S(P^*)$. Then, it is clear that $V^{\mu}$ is unitarily equivalent to $S^\lambda$. The proof of the converse part of Thm. \ref{main intro} is completed by a Zorn's lemma argument.

We also compute the index and the gauge group, i.e. the group of automorphisms, of the product system $E^{(\lambda,k)}$. For $\ell \in \bbn_\infty$, the unitary group of a  Hilbert space of dimension $\ell$ will be denoted by $U(\ell)$. We denote the gauge group of $E^{(\lambda,k)}$ by $G$. 

\begin{theorem}

\label{gauge computation}
    Let $I$ be a non-empty countable set, let $\lambda:I \to S(P^*)$ be an injective map,  and let $k:I \to \bbn_\infty$ be a map. Then, 
    \[
    Ind(E^{(\lambda,k)})=\sum_{i \in I}k_i.
    \]
\begin{enumerate}
    \item If $I=\{i\}$ is singleton, then $E^{(\lambda,k)}$ is isomorphic to the product system of the CCR flow associated with $S^{(\lambda_i,k_{i})}$. In this case, the gauge group $G$ acts transitively on the set of normalised units, and $G$ is isomorphic to the gauge group of the  CCR flow associated with $S^{(\lambda_i,k_{i})}$.
\item If the cardinality of $I$ is at least two, then the gauge group $G$ does not act transitively on the set of normalised units. In this case, $G$ is isomorphic to $\displaystyle \bbr^d \times \prod_{i \in I}U(k_i)$.
\end{enumerate}
\end{theorem}

The organization of this paper is as follows.

After this introductory section, in Section 2, we collect a few definitions concerning product systems, additive decomposable vectors and units.
To keep the paper  self contained, we  give an overview of the results derived in \cite{srinivasanmargett} and in \cite{anbuCAR} concerning the exponential map that plays a crucial role in our analysis. Thm. \ref{oneone} is proved in  Section 3. In Section 4, we prove Thm. \ref{main intro}. In Section 5, we prove Thm.  \ref{gauge computation}. We also prove that the gauge group of a CAR flow need not act transitively on the set of normalised units.

\section{Additive decomposable vectors and the exponential map}

We will make extensive use of the exponential map defined initially for `addits' in \cite{srinivasanmargett} and later extended to `coherent sections of additive decomposable vectors' in \cite{anbuCAR}. To keep the paper fairly self contained, we give a quick overview of the results of \cite{srinivasanmargett} and \cite{anbuCAR}.
 We start by first recalling the definition of  product systems, units, and what it means for a product system to be type I.

Let $P$ be a closed, convex, cone in $\bbr^d$, where $d \geq 1$, which is spanning, i.e. $P-P=\bbr^d$ and pointed, i.e. $P \cap -P=\{0\}$. The letter $P$ is reserved to denote such a cone for the rest of this paper. Let $P_\infty:=P \cup \{\infty\}$. For $x,y \in \bbr^d$, we say $x \leq y$ if $y-x \in P$. For $x \in P$, set 
\[
[0,x]:=\{y \in \bbr^d:0 \leq y \leq x\}.
\]
For $x=\infty$, we let $[0,x]=P$.

Let $E:=\{E(x)\}_{x \in P}$ be a measurable field of non-zero separable Hilbert spaces together with an associative  multiplication defined on the disjoint union $\displaystyle \coprod_{x \in P}E(x)$. Then, $E$ together with the multiplication is called \emph{a product system} if 
the following properties are satisfied.
\begin{enumerate}
\item[(1)] If $u \in E(x)$ and $v \in E(y)$, then $uv \in E(x+y)$.
\item[(2)] For $x,y \in P$, the map 
\[
E(x) \otimes E(y) \ni u \otimes v \to uv \in E(x+y)
\]
is a unitary operator. 
\item[(3)] For measurable sections $r,s,t$, the map 
\[
P \times P \ni (x,y) \to \langle r(x)s(y)|t(x+y) \rangle \in \mathbb{C}\]
is measurable.
\end{enumerate}

Let $E:=\{E(x)\}_{x \in P}$ be a product system. A measurable section $\displaystyle u:P \to \coprod_{x \in P}E(x)$ is called a \emph{unit} if 
\begin{enumerate}
\item[(1)]for $x \in P$, $u_x \neq 0 $, and
\item[(2)] for $x,y \in P$, $u_xu_y=u_{x+y}$.
\end{enumerate}
We denote the set of units of $E$ by $\mathcal{U}_E$.
 We say that $E$ is \emph{spatial} if $\mathcal{U}_E$ is non-empty.

Let $F:=\{F(x)\}_{x \in P}$ be a field of non-zero Hilbert spaces such that, for $x \in P$, $F(x) \subset E(x)$. We say $F$ is a \emph{subsystem of $E$} if  for every $x,y \in P$, \[
F(x+y)=\overline{span\{uv: u \in F(x), v \in F(y)\}}.\]
Let $u \in \mathcal{U}_E$. We say a subsystem $F$ contains $u$ if $u_x \in F(x)$ for every $x \in P$. We say $F$ contains  $\mathcal{U}_E$ if $F$ contains $u$ for every $u \in \mathcal{U}_E$.  The product system $E$ is said to be type I if $E$ is spatial and the only subsystem of $E$ that contains $\mathcal{U}_E$ is $E$.

We  denote the gauge group of $E$, i.e. the group of automorphisms of $E$, by $G_E$. For $u\in\mathcal{U}_{E}$ and $\Psi\in G_{E}$, let $\Psi.u\in \mathcal{U}_{E}$ be given by \[(\Psi.u)_{x}=\Psi_x(u_{x})\] for $x\in P$. A unit $u=\{u_{x}\}_{x\in P}$ of $E$ is said to be \textbf{normalised} if $||u_{x}||=1$ for each $x\in P$. The set of normalised units of $E$ is denoted by $\mathcal{U}^{n}_{E}$. We say that the gauge group acts transitively on the set of normalised units if the action  of $G_E$ on $\mathcal{U}^{n}_E$ given by \[G_{E}\times \mathcal{U}^{n}_{E}\ni (\Psi,u)\to \Psi.u \in \mathcal{U}_E^{n}\] is transitive. 

 In \cite{srinivasanmargett}, Margetts and Srinivasan introduced the notion of addits of a $1$-parameter product system and constructed an exponential map that, after suitable normalisation, sets up a bijective correspondence between 
addits and units. The notion of addits was also considered independentely by Bhat, Lindsay and Mukherjee in \cite{Bhat_Mukherjee}. In \cite{srinivasanccr}, Srinivasan introduced the concept of  additive decomposable vectors. Imitating the techniques of \cite{srinivasanmargett}, Arjunan in \cite{anbuCAR} showed that there is a bijective correspondence between the set of additive decomposable vectors and the set of decomposable vectors. Since this bijection and the exponential map play a key role in what follows, we summarise the main resuls of \cite{srinivasanccr} and \cite{anbuCAR}. 

Let $E=\{E(x)\}_{x \in P}$ be a spatial product system over $P$ with a reference unit $e:=\{e_x\}_{x \in P}$ that is normalised, i.e. $||e_x||=1$ for $x \in P$. Such a pair $(E,e)$  was called a pointed product system in \cite{Bhat_Mukherjee}. The reference unit $e$ is fixed until further mention. 

\begin{definition}[\cite{srinivasanccr}]
    Let $x \in P$, and let $b \in E(x)$. We say that $b$ is an additive decomposable vector if $b \perp e_x$ and for $y \leq x$, there exists $b_y \in E(y)$ and $b(y,x) \in E(y-x)$ (that are necessarily unique) such that 
\begin{enumerate}
\item[(1)] $b_y \perp e_y$,  $b(y,x) \perp e_{x-y}$, and
\item[(2)] $b=b_y e_{x-y}+e_y b(y,x)$.
\end{enumerate}
\end{definition}
For $x \in P$, let \[
\mathcal{A}_e(x):=\{b \in E(x): \textrm{$b$ is additive decomposable}\}.\]

Let $x \in P_\infty$, and let $\{b_y\}_{y \in [0,x]}$ be a family of additive decomposable vectors such that $b_y \in \mathcal{A}_e(y)$ for every $y \in [0,x]$. We call such a family a  coherent section of additive decomposable vectors if for every $y,z \in [0,x]$ with $y\leq z$, there exists, a necessarily unique, $b(y,z) \in E(z-y)$ such that 
\[
b_z=b_y e_{z-y}+e_{y}b(y,z).
\]
A coherent section of additive decomposable vectors $\{b_y\}_{y \in P}$ is called \emph{an addit} if $b(y,z)=b_{z-y}$ whenever $y \leq z$.

Suppose $\{b_y\}_{y \in [0,x]}$ is a coherent section of additive decomposable vectors. 
It is clear that  the collection $\{b(y,z): y,z \in [0,x], y \leq z\}$ satisfies the following propagator equation: for $y_1 \leq y_2 \leq y_3$,
\[
b(y_1,y_3)=b(y_1,y_2)e_{y_3-y_2}+e_{y_2-y_1}b(y_2,y_3).
\]
Let $x \in P$. Given $b \in A_{e}(x)$, it follows from Lemma 3.2 of \cite{srinivasanccr} that there exists a unique coherent section of additive decomposable vectors $\{b_y\}_{y \in [0,x]}$ such that $b_x=b$. 

Next, we recall the definition of decomposable vectors. Let $x \in P$, and let $u \in E(x)$ be a non-zero vector. We say that $u$ is \emph{decomposable} if whenever $y \leq x$, there exists $v \in E(y)$ and $w \in E(x-y)$ such that $u=vw$. For $x \in P$, let \[
D_{e}(x):=\{u \in E(x): \textrm{$u$ is decomposable and $\langle u|e_x \rangle=1$}\}.
\]
Let $x \in P_\infty$, and let $\{u_y\}_{y \in [0,x]}$ be a family of decomposable vectors such that $u_y \in D_{e}(y)$ for every $y \in [0,x]$. We say that $\{u_y\}_{y \in [0,x]}$ is a left coherent section of decomposable vectors if for $y,z \in [0,x]$ with $y \leq z$, there exists a unique $u(y,z) \in D_{e}(z-y)$ such that $u_z=u_yu(y,z)$. 

Next, we recall the definition of the exponential map, in the $1$-parameter setting, that sets up a bijective correspondence between $\mathcal{A}_{e}(\cdot)$ and $D_{e}(\cdot)$. Let $E:=\{E(t)\}_{t \geq 0}$ be a $1$-parameter product system with a reference unit $\{e_t\}_{t \geq 0}$ that is normalised. Suppose $t \geq 0$. Let  $b \in \mathcal{A}_{e}(t)$ be given. Let $\{b_s\}_{s \in [0,t]}$ be the coherent section of additive decomposable vectors such that $b_t=b$.

For every $n \in \{0,1,2,\cdots\}$, define a section $x^{(n)}:[0,t] \to \displaystyle \coprod_{s \in [0,t]}E(s)$ inductively as follows: for $s \in [0,t]$, set $x^{(0)}_s:=e_s$ and $x^{(1)}_s:=b_s$, and for $n \geq 2$, let 
\[
x^{(n)}_s:=\int_{0}^{s}x_{r}^{(n-1)}db_r.
\]
Then, we set 
\[
Exp(b):=\sum_{n=0}^{\infty}x_t^{(n)}.
\]
It was proved in Prop. 3 of \cite{anbuCAR} that the series $\sum_{n=1}^{\infty}x_t^{(n)}$ is norm convergent in $E(t)$.  The integral $\int_{0}^{s}x_{r}^{(n-1)}db_r$ is called It\^{o} integral whose definition is recalled below for the reader's benefit. 

Let $n \geq 2$, and let $s \in [0,t]$ be given. For every $k \geq 1$, partition $[0,s]$ into $k$ intervals of length $\frac{s}{k}$. For $i=0,1,2,\cdots,k-1$, set $r_i^{(k)}:=\frac{is}{k}$. 
Define 
\begin{equation}
\label{ito integral}
S_k:=\sum_{i=0}^{k-1}x_{r_i^{(k)}}^{(n-1)}b(r_{i}^{(k)},r_{i+1}^{(k)})e_{s-r_{i+1}^{(k)}}
\end{equation}
Note that $S_k \in E(s)$ for every $k \geq 1$. Moreover, the sequence $(S_k)_k$ converges in norm whose limit we denote by $\int_{0}^{s}x_r^{(n-1)}db_r$.
 The norm convergence of $(S_k)_k$ was shown in \cite{srinivasanmargett} when $\{b_s\}_{s \geq 0}$ is an addit (see Prop. 5.4 of \cite{srinivasanmargett}). It was observed in \cite{anbuCAR} that the same proof works if we replace an addit by a coherent section of additive decomposable vectors (see Prop. 2 of \cite{anbuCAR}).

\begin{prop}[\cite{srinivasanmargett}, \cite{anbuCAR}]
\label{properties of exponential}
The exponential map satisfies the following key properties. 
\begin{enumerate}
\item[(1)] Let $t\geq 0$. For $b \in \mathcal{A}_e(t)$, $Exp(b) \in \mathcal{D}_e(t)$. Morever, the map 
\[
\mathcal{A}_e(t) \ni b \to Exp(b) \in D_e(t)
\]
is a bijection. For $b_1,b_2 \in \mathcal{A}_e(t)$, 
\[
\langle Exp(b_1)|Exp(b_2) \rangle=exp(\langle b_1|b_2 \rangle).
\]
For $b \in A_{e}(t)$, let $\{b_s\}_{s \in [0,t]}$ be the unique coherent section of decomposable vectors such that $b_t=b$. Then, 
\begin{equation}
\label{exp is decomposable}
Exp(b_r)Exp(b(r,s))=Exp(b_s).
\end{equation}
for $0 \leq r \leq s \leq t$. 

\item[(2)] If $\{b_t\}_{t \in [0,\infty)}$ is an addit, then $\{Exp(b_t)\}_{t \in [0,\infty)}$ is a unit. Conversely, if $\{u_t\}_{t \in [0,\infty)}$ is a unit  and $u_t \in D_e(t)$ for every $t$, then there exists an addit $\{b_t\}_{t \in [0,\infty)}$ such that $u_t=Exp(b_t)$ for every $t \geq 0$.
\end{enumerate}
\end{prop}
For a proof of the above proposition, we refer the reader to \cite{srinivasanmargett} and \cite{anbuCAR}.

\begin{definition}
\label{ray wise exponential}
Let $E:=\{E(x)\}_{x \in P}$ be a spatial product system over $P$ with a normalised reference unit $\{e_x\}_{x \in P}$. For $x \in P$, let $\widetilde{E}:=\{\widetilde{E}(tx)\}_{t \geq 0}$ be the spatial $1$-product system with reference unit $\{\widetilde{e}_{tx}\}_{t \geq 0}$. We denote the exponential map of $\widetilde{E}$ by $Exp_x$. 
    In particular, if $b \in \mathcal{A}_e(x)$, then $Exp_x(b)$ is a well defined vector in $E(x)$.
If $P=[0,\infty)$, then for $s \geq 0$, we omit the subscript `$s$' from $Exp_s$ and simply denote it by $Exp$. 
    
\end{definition}

Suppose $V$ is a pure isometric representation of $P$ on a  separable Hilbert space $\clh$, and let $E^V:=\{E^V(x)\}_{x \in P}$ be the product system of the CAR flow $\beta^V$. Recall that, for $x \in P$, $E^V(x)=\Gamma_a(Ker (V_x^*))$ and the multiplication rule is given by Eq. \ref{multiplication}. For $x \in P$, let  $\Omega_x$ be the vacuum vector of $\Gamma_a(Ker(V_x^*))$. Then, $\Omega:=\{\Omega_x\}_{x \in P}$ is a unit which we call the vacuum unit. We always consider the \textbf{vacuum unit as the reference unit} of $E^V$ while considering the exponential map. The set of additive decomposable vectors were determined by Srinivasan in \cite{srinivasanccr}, and we summarise the results in the next remark. A couple of definitions are in order before we make the remark.

For $x \in P$, we denote the range projection of $V_x$ by $E_x$.
\begin{definition}
    A map $\xi:P \to \clh$ is called an additive cocycle  if for $x \in P$, $E_x\xi_x=0$, and for $x,y \in P$, \[
    \xi_{x+y}=\xi_x+V_x\xi_y.
    \]
    The set of additive cocycles  is denoted by $\mathcal{A}(V)$.

    Let $u:=\{u_x\}_{x \in P}$ be a unit of $E^V$. We say that $u$ is an \textbf
{exponential unit} if $u_x \in \mathcal{D}_{\Omega}(x)$ for every $x \in P$, i.e. $\langle u_x|\Omega_x\rangle=1$ for $x \in P$. 
\end{definition}

\begin{remark}\label{rmk}
\label{Additive decomposable vectors}
With the foregoing notation, we have the following.
\begin{enumerate}
\item Let $x \in P$. Then, it follows from Prop. 4.1 of \cite{srinivasanccr} that $\mathcal{A}_{\Omega}(x)=Ker(V_x^*)$. It is easy to see that $Ker(V_x^*) \subset \mathcal{A}_{\Omega}(x)$. To see this suppose $x \in P$ and $\xi \in Ker(V_x^{*})$. Let $y \in P$ be such that $y \leq x$. Set $\xi_y:=E_y^{\perp}\xi$ and $\xi(y,x):=V_y^*\xi$.
Then, it is clear from the definition of the multiplication rule that 
\begin{equation}
\label{kernel is additive}
\xi_y \cdot \Omega_{x-y}+\Omega_y \cdot \xi(y,x)=\xi.
\end{equation}

Suppose  $P=[0,\infty)$, $t \geq 0$ and $\xi \in Ker(V_t^*)$. It follows from  Eq. \ref{exp is decomposable} that  
\[
Exp(E_s^\perp \xi)Exp(V_s^*\xi)=Exp(\xi)
\]
for $s \leq t$. Equivalently, if $P=[0,\infty)$, then for $s,t \geq 0$, $\xi \in Ker(V_s^*)$ and $\eta \in Ker(V_t^*)$,
\begin{equation}
\label{ccr is car}
Exp(\xi)Exp(\eta)=Exp(\xi+V_s\eta).
\end{equation}

\item Let $\xi:P \to \mathcal{H}$ be a map such that $\xi_x \in Ker(V_x^{*})$ for every $x \in P$. It is clear from the definition of an addit, Eq. \ref{kernel is additive} and Eq. \ref{multiplication} that $\{\xi_x\}_{x \in P}$ is an addit if and only if $\xi$ is an additive cocycle for $V$. 
\item Let $u$ be an exponential unit of $E^V$. Let $x \in P$. 
Note that $\{u_{tx}\}_{t \geq 0}$ is a unit for the one parameter product system $\{E(tx)\}_{t \geq 0}$. It follows from Prop. \ref{properties of exponential} that there exists a unique $\xi_x \in Ker(V_x^*)$ such that $u_x=Exp_x(\xi_x)$. It follows from the definition of $Exp_x$ that the $1$-particle vector of $u_x$ is $\xi_x$.
Comparing the $1$-particle vectors of the equation 
\[
Exp_{x+y}(\xi_{x+y})=u_{x+y}=u_x \cdot u_y=Exp_x(\xi_x) \cdot Exp_y(\xi_y)\]
we see that for $x,y \in P$, $\xi_{x+y}=\xi_x+V_x\xi_y$, i.e. $\xi:=\{\xi_x\}_{x \in P}$ is an additive cocycle for $V$. 

Thus, if $u$ is an exponential unit, then $u$ is of the form $\{Exp_x(\xi_x)\}_{x \in P}$ for a unique  $\xi \in \mathcal{A}(V)$. However, it is not necessary  that  
$Exp(\xi):=\{Exp_x(\xi_x)\}_{x \in P}$ is a unit if $\xi \in \mathcal{A}(V)$ (see Prop. \ref{additive cocycles that give units}). 
\end{enumerate}
 \end{remark}

 \section{Injectivity of the CAR functor} 
In this section, we prove Thm. \ref{oneone}.  Suppose $V$ is a pure isometric representation of $[0,\infty)$. Let $F:=\{F(t)\}_{t \geq 0}$ be the product system of the associated CCR flow $\alpha^V$. Recall that for $t \geq 0$, $F(t):=\Gamma_s(Ker(V_t^*))$ and the multiplication rule is given by
 \begin{equation}
 \label{ccr formula}
 e(\xi)e(\eta)=e(\xi+V_s\eta)\end{equation}
 for $\xi \in Ker(V_s^*)$ and $\eta \in Ker(V_t^*)$. Here, $\{e(\xi):\xi \in Ker(V_s^*)\}$ denotes the set of  exponential vectors.
The vacuum unit of $F$ is denoted by $\Omega:=\{\Omega_t\}_{t \geq 0}$. We use the same letter $\Omega$ to denote the vacuum unit of both CCR and CAR flows. 
 Let  $E$  denote  the product system of the CAR flow $\beta^V$ with the reference unit $\Omega$. For $t\geq 0$, let $D^{E}(t)$ and $D^{F}(t)$ denote the decomposable vectors in $E(t)$ and $F(t)$ respectively.
 
  For $t\geq 0$, let  \[D_{\Omega}^{E}(t):=\{u\in D^{E}(t): \textrm{$\langle u|\Omega_t\rangle=1$}\}\] and $D_{\Omega}^{F}(t)$ is defined similarly. 
  From Remark \ref{Additive decomposable vectors} and Prop. \ref{properties of exponential}, we have \[D_{\Omega}^{E}(t)=\{Exp(\xi):\textrm{$\xi\in Ker(V_{t}^{*}) $}\},\]
  and from Prop. 2.2 of \cite{SUNDAR} we have \[D_{\Omega}^{F}(t)=\{e(\xi): \textrm{ $\xi\in Ker(V_{t}^{*}$}\}.\] 
It is known that $E$ and $F$ are isomorphic, a fact first proved by Robinson and Powers in \cite{Powers_Rob}. For our purposes, we need the following coordinate free isomorphism.

\begin{lemma} \label{isoccrcar}For each $t\geq 0$, the map $\Psi_{t}:D_{\Omega}^{E}(t)\to D_{\Omega}^{F}(t)$ defined by \[D_{\Omega}^{E}(t)\ni Exp(\xi)\to e(\xi)\in D_{\Omega}^{F}(t)\] is a bijection. Moreover, 
\[\Psi_{s}Exp(\xi).\Psi_{t}Exp(\eta)=\Psi_{s+t}(Exp(\xi).Exp(\eta))\] for $s,t \geq 0$,  $\xi\in Ker(V^{*}_{s})$ and $\eta\in Ker(V^{*}_{t})$.
  The map $\Psi_{t}$ extends uniquely to a unitary again denoted $\Psi_t: E(t)\to F(t)$ for $t\geq 0$. The field of maps $\Psi:=\{\Psi_{t}\}_{t\geq 0}$ is an isomorphism from $E$ onto $F$. \end{lemma}
\begin{proof}
 From Corollary 6.8.3 of \cite{arveson}, $D_{\Omega}^{E}(t)$ is total in $E(t)$, and $D_{\Omega}^{F}(t)$ is total in $F(t)$ for $t\geq 0$. 
Also, by Prop. \ref{properties of exponential}, \[\langle Exp(\eta)|Exp(\xi)\rangle=e^{\langle \eta|\xi \rangle}=\langle e(\eta)|e(\xi)\rangle\] for $\eta,\xi\in Ker(V_{t}^{*})$ for $t\geq 0$. Hence, $\Psi_{t}$ can be extended to a unitary operator, again denoted $\Psi_t$,  $\Psi_{t}:E(t)\to F(t)$, for $t\geq 0$.
It follows from Eq. \ref{ccr is car} and Eq. \ref{ccr formula} that $\Psi:=\{\Psi_t\}_{t \geq 0}:E \to F$ is an isomorphism.  Hence the proof.
\end{proof}
Suppose $V^{(1)}$ and $V^{(2)}$ are pure isometric representations of $[0,\infty)$ on $\mathcal{H}$ and $\mathcal{K}$ respectively.  Denote the product systems of the respective CAR flows by $E_{1}$ and $E_{2}$. 
Let \[\mathscr{U}(V^{(1)}, V^{(2)}):=\{U:\mathcal{H}\to\mathcal{K}:\textrm{ $U$ is a unitary, $UV^{(1)}_{t}=V^{(2)}_{t}U$, $UV^{(1)*}_{t}=V^{(2)*}_{t}U$ for $t\geq 0$}\}.\]
The proof of the next lemma is essentially an application of   Lemma \ref{isoccrcar} and the  gauge group computation of CCR flows due to Arveson. 

  \begin{lemma}\label{isostructure} 
  Suppose $\Psi: E_{1}\to E_{2}$ is an isomorphism of product systems. 
Then, there exists $U\in\mathscr{U}(V^{(1)}, V^{(2)})$, $\xi=\{\xi_{t}\}_{t\geq 0}\in\mathcal{A}(V^{(2)})$, and $\lambda\in\mathbb{R}$ such that

  \[\Psi_{t}(Exp(\eta))=e^{i\lambda t}e^{\frac{-||\xi_{t}||^{2}}{2}-\langle U\eta|\xi_{t}\rangle }Exp(U\eta+\xi_{t})\] for $\eta\in Ker(V^{(1)*}_{t})$, $t\geq 0$.
  \end{lemma}
  \begin{proof} We denote by $F_{1}$ and $F_{2}$ the product systems of the CCR flows associated with $V^{(1)}$ and $V^{(2)}$ respectively. For  $i=1,2$, let $\Theta_{i}=\{\Theta_{i}(t)\}_{t \geq 0}$ be the isomorphism from $E_{i}$ onto $F_{i}$ given by 
  \[\Theta_{i}(t)Exp(\eta)=e(\eta)\] for $\eta\in Ker(V^{(i)*}_{t})$ and $t\geq 0$. The isomorphism $\Theta_i$ for $i=1,2$ is guaranteed by Lemma \ref{isoccrcar}. 
  Then, $\Delta=\Theta_{2}\circ\Psi\circ \Theta_{1}^{-1}$ is an isomorphism from $F_{1}$ onto $F_{2}$.

  It follows from Corollary 2.6.10 of \cite{arveson} that  $V^{(1)}$ and $V^{(2)}$ are unitarily equivalent. Suppose $W:\mathcal{H}\to\mathcal{K}$ is a unitary such that $WV^{(1)}_{t}=V^{(2)}_{t}W$ for $t\geq 0$. For $t\geq 0$, let $\Lambda_t:F_1(t) \to F_2(t)$ be the unitary operator such that  \[\Lambda_{t} exp(\xi)=exp(W\xi)\] 
  for $\xi \in Ker(V^{(1)*}_t)$. Clearly,  $\Lambda:=\{\Lambda_{t}\}_{t\geq 0}:F_{1}\to F_{2}$ is an isomorphism.

 Now consider the map $T:=\Lambda^{-1}\circ\Delta$.  The map $T$ is an automorphism of $F_{1}$. By Thm. 3.8.4 of \cite{arveson}, there exists $U \in \mathscr{U}(V^{(1)},V^{(1)})$, an additive cocycle
 $\xi=\{\xi_{t}\}_{t\geq 0}$  for $V^{(1)}$ and $\lambda \in \bbr$ such that 
   \[T_{t} e(\eta)=e^{i\lambda t}e^{-\frac{||\xi_{t}||^{2}}{2}-\langle U\eta|\xi_{t}\rangle}e(U\eta+\xi_{t})\] 
 for $\eta \in Ker(V_t^{(1)*})$.
  For $t \geq 0$, set $\widetilde{\xi}_t:=W\xi_t$ and $\widetilde{U}:=WU$. Then, $\widetilde{U} \in \mathscr{U}(V^{(1)},V^{(2)})$ and $\widetilde{\xi}$ is an additive cocycle for $V^{(2)}$. 

  Since $\Delta:=\Lambda \circ T$, it follows that for   
   for $t\geq 0$ and $\eta\in Ker(V^{(1)*}_{t})$ ,
   \begin{align*}
\Delta_{t}e(\eta)&=e^{i\lambda t}e^{-\frac{||\xi_{t}||^{2}}{2}-\langle U\eta|\xi_{t}\rangle}e(WU\eta+W\xi_{t})\\
&=e^{i\lambda t}e^{-\frac{||W\xi_{t}||^{2}}{2}-\langle WU\eta|W\xi_{t}\rangle}e(WU\eta+\widetilde{\xi}_{t})\\
&=e^{i \lambda t}e^{-\frac{||\widetilde{\xi}_t||^{2}}{2}-\langle \widetilde{U}\eta|\widetilde{\xi}_t\rangle}e(\widetilde{U}\eta+\widetilde{\xi}_t).
\end{align*}

 Since $\Psi=\Theta_2^{-1} \circ \Delta \circ \Theta_1$, it is  immediate that
  \[\Psi_{t} Exp(\eta)=e^{i\lambda t}e^{\frac{-||\widetilde{\xi}_{t}||^{2}}{2}-\langle \widetilde{U}\eta|\widetilde{\xi}_{t}\rangle }Exp(\widetilde{U}\eta+\widetilde{\xi}_{t})\] for $\eta\in Ker(V^{(1)*}_{t})$ and $t\geq 0$. Hence the proof.
    \end{proof}
  
Let $P$ be a closed convex cone in $\mathbb{R}^d$ which is spanning and pointed.    Let $V^{(1)}$, $V^{(2)}$ be two pure isometric representations of $P$ on  Hilbert spaces $\mathcal{H}_1$ and $\mathcal{H}_2$ respectively. We denote the product systems of the corresponding CAR flows by $E^{(1)}$ and $E^{(2)}$ respectively.
\begin{theorem} \label{injectivityCAR} The product systems $E^{(1)}$ and $E^{(2)}$ are isomorphic if and only if $V^{(1)}$ and $V^{(2)}$ are unitarily equivalent.\end{theorem}
\begin{proof}  
Suppose $E^{(1)}$ and $E^{(2)}$ are isomorphic. Let $\Psi=\{\Psi_{a}\}_{a\in P}:E^{(1)}\to E^{(2)}$ be an isomorphism. Let $a\in Int(P)$. Consider the one parameter product system  $E_{a}^{(1)}:= \{E_{a}^{(1)}
(t)\}_{t \geq 0}$, where for $t \geq 0$, $E_{a}^{(1)}(t):=E^{(1)}(ta)$. Similarly,  consider the one parameter product system $E_{a}^{(2)}:= \{E_{a}^{(2)}(t)\}_{t \geq 0}$, where $E_{a}^{(2)}(t):=E^{(2)}(ta)$ for $t\geq 0$. 
Then, $\{\Psi_{ta}\}_{t \geq 0}:E_a^{(1)} \to E_a^{(2)}$ is an isomorphism.

By Lemma \ref{isostructure}, there exists a unitary $U_{a}:\mathcal{H}_1\to\mathcal{H}_2$ intertwining $\{V^{(1)}_{ta}\}_{t\geq 0}$ and $\{V^{(2)}_{ta}\}_{t\geq 0}$, an additive cocycle  $ \{\xi^{a}_{t}\}_{t\geq 0}$ of $\{V^{(2)}_{ta}\}_{t\geq 0}$ and $\lambda_{a}\in\mathbb{R}$ such that $\Psi_{ta}$ is of the form 
\begin{equation}\label{raywiseiso1}\Psi_{ta}Exp_{a}(\xi)=e^{i\lambda_{a}t} e^{-\frac{||\xi^{a}_{t}||^{2}}{2}}e^{-\langle U_{a}\xi|\xi^{a}_{t}\rangle}Exp_{a}(U_{a}\xi+\xi^{a}_{t}),
\end{equation} for $\xi\in Ker(V^{(1)*}_{ta}) $ and $t >0$. We will denote $\xi^{a}_{1}$ by $\xi_{a}$.
Hence, 
\begin{equation}
    \label{raywiseiso}
    \Psi_{a}Exp_{a}(\xi)=e^{i\lambda_{a}} e^{-\frac{||\xi_s||^{2}}{2}}e^{-\langle U_{a}\xi|\xi_a\rangle}Exp_{a}(U_{a}\xi+\xi_a)
\end{equation}
for $a \in Int(P)$ and $\xi \in Ker(V^{(1)*}_a)$. 

Let $a, b\in Int(P)$. Since $\Psi $ is an isomorphism, we have
\begin{equation*}\Psi_{a}Exp_{a}(\xi). \Psi_{b}Exp_{b}(\eta)=\Psi_{a+b}(Exp_{a}(\xi).Exp_{b}(\eta)),
\end{equation*} for $\xi\in Ker(V^{(1)*}_{a}) $ and $\eta\in Ker(V^{(1)*}_{b}) $. 

In particular,
 \begin{equation*}\Psi_{a+b}(Exp_{a}(0).Exp_{b}(0))=\Psi_{a}Exp_{a}(0).\Psi_{b}Exp_{b}(0),
 \end{equation*} 
 i.e.
 \[\Psi_{a+b}Exp_{a+b}(0)=e^{i(\lambda_a+\lambda_b)}e^{-\frac{(||\xi_{a}||^{2}+||\xi_{b}||^{2})}{2}}Exp_{a}(\xi_{a}).Exp_{b}(\xi_{b})\]
 \begin{equation} \label{inj1}
      e^{i\lambda_{a+b}}e^{-\frac{||\xi_{a+b}||^{2}}{2}}Exp_{a+b}(\xi_{a+b})=e^{i(\lambda_{a}+\lambda_{b})}e^{-\frac{(||\xi_{a}||^{2}+||\xi_{b}||^{2})}{2}}Exp_{a}(\xi_{a}).Exp_{b}(\xi_{b}).
 \end{equation}
 
Comparing the $0$-particle vectors in LHS and RHS of Eq. \ref{inj1}, we have
 \begin{equation}
     \label{scalar_one}
      e^{i\lambda_{a+b}-\frac{||\xi_{a+b}||^{2}}{2}}= e^{i(\lambda_{a}+\lambda_{b})-\frac{(||\xi_{a}||^{2}+||\xi_{b}||^{2})}{2}}.
 \end{equation}
Now, comparing the $1$-particle vectors of LHS and RHS in Eq. \ref{inj1},  we have 
 \begin{equation}
 \label{additive in injectivity}
 \xi_{a+b}=\xi_{a}+V^{(2)}_{a}\xi_{b}.
 \end{equation}
 
  Let $\xi\in Ker(V^{(1)*}_{a})$.
Consider the equation
\begin{equation}\label{xi}\Psi_{a+b}\xi =\Psi_{a}\xi. \Psi_{b}Exp_{b}(0).
\end{equation}

Considering the exponential map in the one parameter product system $E^{(1)}_a$, we have
\[Exp_{a}(\xi)=\sum_{n=0}^{\infty}x_{1}^{(n)}\] where for $r \in (0,1]$, $x_r^{(0)}=\Omega_{ra}$, $x_r^{(1)}= (1-V_{ra}^{(1)}V_{ra}^{(1)*})\xi$ and $x_r^{(n)}=\int_{0}^{r}x_t^{(n-1)}d\xi_{t}$. Here, for $t \in (0,1]$, $\xi_{t}=(1-V^{(1)}_{ta}V^{(1)*}_{ta})\xi$. 

Let $s \in \bbr$. Suppose 
\[Exp_{a}(s\xi)=\sum_{n=0}^{\infty}y_1^{(n)},\] where the summands are obtained via It\^{o} integration. 
A moment's reflection on the definition of the It\^{o} integral reveals that
\[y_{1}^{(n)}=s^{n}x_1^{(n)}\] for $n=0,1,2,\cdots$.

 Note that \begin{equation}\label{powerseries1}\Psi_a Exp_{a}(t\xi)=\sum_{n=0}^{\infty}\Psi_a y_1^{(n)}=\sum_{n=0}^{\infty}t^{n}\Psi_a x_1^{(n)}\end{equation} for $t \in \bbr$. Since the power series in Eq. \ref{powerseries1} is norm convergent for every $t \in \bbr$, it may be differentiated term by term, and the derivative is given by 
  \[\label{powerseries2}\frac{d}{dt}\Psi_a Exp_{a}(t\xi)=\sum_{n=1}^{\infty}nt^{n-1}\Psi_a( x_1^{(n)}).\]
  
Hence, 
 \begin{equation}
     \label{derivative}\Psi_{a}(\xi)=\frac{d}{dt}\Big|_{t=0}\Psi_{a}Exp_{a}(t\xi).
 \end{equation}
From Eq. \ref{raywiseiso} we have,
\begin{align*}
\frac{d}{dt}\Psi_{a}Exp_{a}(t\xi)
&=\frac{d}{dt} e^{i\lambda_{a}}e^{-\frac{||\xi_{a}||^{2}}{2}}e^{-t\langle U_{a}\xi|\xi_{a}\rangle}Exp_{a}(tU_{a}\xi+\xi_{a})\\
&=e^{i\lambda_{a}}e^{-\frac{||\xi_{a}||^{2}}{2}}\Big(Exp_{a}(tU_{a}\xi+\xi_{a})\frac{d}{dt}e^{-t\langle U_{a}\xi|\xi_{a}\rangle}+ e^{-t\langle U_{a}\xi|\xi_{a}\rangle}\frac{d}{dt}Exp_{a}(tU_{a}\xi+\xi_{a})\Big)\\
&=e^{i\lambda_{a}-\frac{||\xi_{a}||^{2}}{2}}e^{-t\langle U_{a}\xi|\xi_{a}\rangle}\Big(-\langle U_{a}\xi|\xi_{a}\rangle Exp_{a}(tU_{a}\xi+\xi_{a})+U_{a}\xi+q_{a}(t,\xi)\Big)
\end{align*} where the projection of $q_{a}(t,\xi)$ onto $0$-particle and $1$-particle space is zero.
By Eq. \ref{derivative},
\[\label{iso xi}\Psi_{a}(\xi)=\frac{d}{dt}\Big|_{t=0}\Psi_{a}Exp_{a}(t\xi)=e^{i\lambda_{a}-\frac{||\xi_{a}||^{2}}{2}}(-\langle U_{a}\xi|\xi_{a}\rangle Exp_{a}(\xi_{a})+U_{a}\xi+q_{a}(0,\xi)).\]

Similarly,
\begin{align*}\Psi_{a+b}(\xi)=e^{i\lambda_{a+b}-\frac{||\xi_{a+b}||^{2}}{2}}(-\langle U_{a+b}\xi|\xi_{a+b}\rangle Exp_{a+b}(\xi_{a+b})+U_{a+b}\xi+q_{a+b}(0,\xi))
\end{align*} 
where the projection of $q_{a+b}(0,\xi)$ onto $0$-particle and $1$-particle space is zero. 

Therefore, Eq. \ref{xi} implies,
\begin{align*}\label{inj2}&e^{i(\lambda_{a}+\lambda_{b})-\frac{(||\xi_{a}||^{2}+||\xi_{b}||^{2})}{2}}\Big(-\langle U_{a}\xi|\xi_{a}\rangle Exp_{a}(\xi_{a})+U_{a}\xi+q_{a}(0,\xi)\Big).Exp_{b}(\xi_{b})\\&=e^{i\lambda_{a+b}-\frac{||\xi_{a+b}||^{2}}{2}}\Big(-\langle U_{a+b}\xi|\xi_{a+b}\rangle Exp_{a+b}(\xi_{a+b})+U_{a+b}\xi+q_{a+b}(0,\xi)\Big).
\end{align*}
 Using Eq. \ref{scalar_one}, the above equation reduces to 
 \begin{equation}
\label{inj2}
\big(-\langle U_{a}\xi|\xi_{a}\rangle Exp_{a}(\xi_{a})+U_{a}\xi+q_{a}(0,\xi)\Big).Exp_{b}(\xi_{b})=-\langle U_{a+b}\xi|\xi_{a+b}\rangle Exp_{a+b}(\xi_{a+b})+U_{a+b}\xi+q_{a+b}(0,\xi).
 \end{equation}
 By equating the $0$-particle vectors in the above equation,  we  see that
 \begin{equation}\label{adhoc label}\langle U_{a}\xi|\xi_{a}\rangle=\langle U_{a+b}\xi|\xi_{a+b}\rangle. \end{equation}
  Equating the $1$-particle vectors in Eq. \ref{inj2}, we have
 \begin{align*}
 -\langle U_{a}\xi|\xi_{a}\rangle (\xi_a+V^{(2)}_a\xi_b)+U_{a}\xi=-\langle U_{a+b}\xi|\xi_{a+b}\rangle\xi_{a+b}+U_{a+b}\xi.
 \end{align*} Hence, by Eq. \ref{additive in injectivity} and by Eq. \ref{adhoc label}, 
 \begin{align*}U_{a}\xi=U_{a+b}\xi.
 \end{align*}
 Thus, if $\xi\in Ker( V^{(1)*}_{a})$ and $\xi\in Ker( V^{(1)*}_{b})$ for $a,b\in Int(P)$, $U_{a}\xi=U_{a+b}\xi=U_{b}\xi$. Since $ \bigcup_{a \in Int(P)}Ker(V_a^{(1)*})$ is dense in $\clh_1$, it follows that  there exists a unitary operator $U: \mathcal{H}_1\to\mathcal{H}_2$ such that  
 \[
 U\xi=U_a\xi
 \]if  $\xi \in Ker(V^{(1)*}_{a})$. 

 Similarly, considering the equation
  \begin{equation*}\Psi_{a+b}(Exp_{a}(0).\eta)=\Psi_{a}(Exp_{a}(0)).\Psi_{b}\eta
 \end{equation*} for $\eta \in Ker(V^{(1)*}_{b})$ and comparing $0$ and $1$-particle vectors as before, it is not difficult to see  that for $a,b \in Int(P)$ and $\eta \in Ker(V_b^{(1)*})$
 \begin{equation*} UV^{(1)}_{a}\eta=V^{(2)}_{a}U\eta.
 \end{equation*} 
 Since $\displaystyle \bigcup_{b \in Int(P)}Ker(V_b^{(1)*})$ is dense in $\clh_1$, it follows that $UV^{(1)}_a=V^{(2)}_aU$ for every $a \in Int(P)$. 
 As $Int(P)$ is dense in $P$, $U$ intertwines $V^{(1)}$ and $V^{(2)}$.  This completes the proof.
\end{proof}

\section{Type I CAR Flows associated with Isometric Representations with Commuting Range Projections}

In this section, we characterise pure isometric representations of $P$ with commuting range projections that give rise to type I CAR flows. Recall that 
a pure isometric representation $V$ of $P$ is said to have \textbf{commuting range projections} if  $\{V_{a}V_{a}^{*}: a \in P\}$ is a commuting family of projections. First, we recall the dynamical system that encodes all pure  isometric representations with commuting range projections. 
Let $\mathcal{C}(\mathbb{R}^{d})$ denote the set of closed subsets of $\mathbb{R}^{d}$ which we equip with the Fell topology. 
Let
\begin{align*}Y_{u}:&=\{A\in\mathcal{C}(\mathbb{R}^{d}): \textrm{ $-P+A\subset A$, $A\neq \emptyset$, $A\neq \bbr^d$}\},\\
X_{u}:&=\{A\in Y_{u}:\textrm{ $0\in A$}\}. \end{align*}
The space $Y_{u}$ is  locally compact, Hausdorff and second countable on which the group $\bbr^d$ acts.  The action is given by  the map
\[\mathbb{R}^{d}\times Y_{u}\ni (x,A)\to A+x\in Y_{u}.\]
Note that $X_u+P \subset X_u$ and $Y_u=\displaystyle \bigcup_{a \in P}(X_u-a)$. 
It was proved in \cite{Sundar_NYJM} that there is a bijective correspondence between the class of pure isometric representations of $P$ and the class of covariant representations of the dynamical system $(Y_u,\bbr^d)$. 

For us, the most important examples  of isometric representations are the `shift semigroups on $X_u$' that arise from invariant measures on $Y_u$. 
Let $\mu$ be a non-zero Radon measure on $Y_{u}$ which is also invariant under the above mentioned action of $\mathbb{R}^{d}$.  Define an isometric representation $V^{\mu}$ of $P$ on $L^{2}(X_{u},\mu)$ by
\[V^{\mu}_{a}f(A):=\begin{cases}f(A-a) & \textrm{~if~}A-a\in X_{u},\\
0 & \textrm{~if~}A-a \notin X_u.
\end{cases}\]
Clearly, $V^{\mu}$ has commuting range projections.
 
It follows from Lemma 3.8 and Lemma 3.9 of \cite{sundarkms} that the set $X_{u}\setminus (X_{u}+a)$ has compact closure for $a \in P$ and has positive measure if $a \in Int(P)$. It is now clear that $\{1_{X_{u}\setminus (X_{u}+a)}\}_{a\in P}$ is a non-zero additive cocycle for $V^{\mu}$. 

Recall that $P^*$ stands for the dual cone of $P$,  and $S(P^*):=\{\lambda \in P^*: ||\lambda||=1\}$. Suppose $\lambda\in S(P^{*})$. Define
\[H^{\lambda}:=\{x\in\mathbb{R}^{d}: \textrm{$\langle\lambda|x\rangle\leq 0$}\}.\] 
Define \[Y_{u}^{\lambda}:=\{H^{\lambda}+t\lambda:\textrm{ $t\in\mathbb{R}$}\}=\{H^{\lambda}+z: z \in \bbr^d\},\] and \[X_{u}^{\lambda}:=\{H^{\lambda}+t\lambda:\textrm{ $t\geq 0$}\}=Y_u^{\lambda} \cap X_u.\] Note that $X_{u}^{\lambda}$, $Y_{u}^{\lambda}\subset Y_{u}$ are closed. Also, observe that $Y_u^{\lambda}$ is $\bbr^d$-invariant and $X_u^{\lambda}+P \subset P$. 

Let $S:=\{S_{t}\}_{t\geq 0}$ be the one parameter shift semigroup on $L^{2}[0,\infty)$ defined by 
\[S_tf(x)=f(x-t)1_{[0,\infty)}(x-t).\]
For $\lambda\in S(P^{*})$, define the isometric representation $S^{\lambda}$ of $P$ on $L^{2}[0,\infty)$ by 
\[S^{\lambda}_{a}= S_{\langle\lambda|a\rangle}\] for $a\in P$.

  Let $\mu$ be a non-zero $\bbr^d$-invariant Radon measure on $Y_u$.   Denote the product system of the CAR flow associated with $V^{\mu}$ by $E^{\mu}$. For $a\in P$, let $\eta_{a}:=1_{X_{u}\setminus X_{u}+a}$ , let $\eta:=\{\eta_{a}\}_{a\in P}$, and let $Exp(\eta):=\{Exp_a(\eta_a)\}_{a \in P}$. 
\begin{lemma}\label{<-}
Suppose $supp(\mu)=Y^{\lambda}_{u}$ for some $\lambda \in S(P^*)$. Then, 
\begin{enumerate}[(i)]\item $V^{\mu}$ is unitarily equivalent to $S^{\lambda}$, and
\item $Exp(\eta)$ is a unit for $E^{\mu}$.\end{enumerate}
\end{lemma}
\begin{proof}   We first show that $V^{\mu}$ is unitarily equivalent to $S^{\lambda}$.
The reader may verify that the map 
$\mathcal{F}: Y^{\lambda}_{u}\to \mathbb{R}$ defined by
\[\mathcal{F}(H^{\lambda}+t\lambda)=t\] is a homeomorphism, and $\mathcal{F}(X_u^{\lambda})=[0,\infty)$. Also, \[\mathcal{F}((H_\lambda+t\lambda)+a)=t+\langle \lambda|a \rangle\]
for $t \in \bbr$ and $a \in \bbr^d$. Since $\mu$ is an invariant measure, the push forward measure $\mathcal{F}_{*}\mu$ is the Lebesgue measure, denoted by $m$. 
  Define a unitary operator  $U: L^{2}([0,\infty),m)\to L^{2}(X_{u}^{\lambda},\mu)$ by
 \[Uf=f \circ \mathcal{F}.\] It is routine to verify that $U$ intertwines $S^{\lambda}$ and $V^{\mu}$. Note that $U^{-1}(1_{X_{u}\setminus X_{u}+a})=1_{(0,\langle \lambda|a\rangle)}$ for $a\in P$. By abusing notation, we may assume that $V^{\mu}=S^{\lambda}$ and $\eta_{a}=1_{(0,\langle \lambda|a\rangle)}$ for $a\in P$.

Let $Exp$ denote the exponential map of  the one parameter product system $\widetilde{E} $ of the CAR flow associated with $\{S_{t}\}_{t\geq 0}$.  Let $u_{t}=Exp(1_{(0,t)})$ for $t\geq 0$. It follows from Remark \ref{Additive decomposable vectors} and Prop. \ref{properties of exponential} that $u=\{u_t\}_{t \geq 0}$ is unit for $\widetilde{E}$.
Note that $Exp_{a}(\eta_{a})=u_{\langle\lambda|a\rangle}$ for $a\in P$. Observe that
\[Exp_{a}(\xi_{a}).Exp_{b}(\xi_{b})=u_{\langle\lambda|a\rangle}.u_{\langle\lambda|b\rangle}=u_{\langle\lambda|a+b\rangle}=Exp_{a+b}(\xi_{a+b}).\]  for $a, b\in P$, i.e. $Exp(\eta)$ is a unit.
\end{proof}
The following picture of $Y_{u}$ will be useful in the sequel. 
Suppose $\{v_{i}:\textrm{ $i=1,2,\cdots , d$}\}$ is a basis for $\mathbb{R}^{d}$ such that $v_{i}\in Int(P)$ for $i=1,2,\cdots , d$. For $i\in \{1,2,\cdots , d\}$, let \[Q_{i}:=span\{v_{j}: \textrm{$j\in \{1,2,\cdots , d\}$, $j\neq i$}\}.\]
Let $f:Q_{i}\to\mathbb{R}$ be continuous,  and let $a=\sum_{i=1}^{d}a_{i}v_{i}\in \bbr^d$. Define 
\[\Big(\Psi_{i}(a) f\Big)(\sum_{ j\neq i}x_{j}v_{j})=f(\sum_{j\neq i}(x_{j}-a_{j})v_{j})+a_{i}.\]
 We say that $f$ is decreasing if $\Big(\Psi_{i}(a)f\Big)(x)\leq f(x)$ for $x\in Q_{i}$ and $a\in -P$.
  Define \[\mathcal{F}_{i}:=\{f:Q_{i}\to\mathbb{R} : \textrm{ $f$ is continuous and decreasing}\}.\] 

Suppose $A\in Y_{u}$. Fix $i\in \{1,2,\cdots ,d\}$. By Lemma 4.3 of \cite{anbuCAR}, there exists a continuous   function $f^{A}_{i}$ such that 
\[A=\{\sum_{j=1}^{d}x_{j}v_{j}\in \mathbb{R}^{d}:\textrm{ $x_{i}\leq f^{A}_{i}( \sum_{k\neq i}x_{k}v_{k})$}\}.\]
For $a=\displaystyle{\sum_{j=1}^{d}}a_{j}v_{j}\in \bbr^{d}$, note that $f_i^{A+a}=\Psi_i(a)f_i^A$. Since $-P+A \subset A$, $f_i^A \in \mathcal{F}_i$.  For $A,B\in Y_{u}$, $f^{A}_{i}=f^{B}_{i}$ if and only if $A=B$. 

Conversely, if $f:Q_{i}\to\mathbb{R}$ is continuous and decreasing, the set
\[A_{f}:=\{\sum_{j=1}^{d}x_{j}v_{j}\in \mathbb{R}^{d}:\textrm{$x_{i}\leq f(\sum_{ k\neq i}^{d}x_{k}v_{k})$ }\}\] is in $Y_{u}$. Let $i\in \{1,2,\cdots d\}$. Clearly,  \[X_{u}=\{A\in Y_{u}: \textrm{$f^{A}_{i}(0)\geq 0$}\}.\] 

\begin{prop}
    
\label{pointwise convergence} Suppose $(A_{n})_{n\in\mathbb{N}}$ is a sequence in $Y_{u}$ such that $A_{n}\to A$ in $Y_{u}$. Then, the sequence $\{f^{A_{n}}_{i}\}_{n\in\mathbb{N}}$ converges pointwise to $f^{A}_{i}$ for each $i\in\{1,2,\cdots , d\}$.
\end{prop}
\begin{proof} Fix $i\in \{1,2,\cdots ,d\}$.  Suppose $(A_{n})$ is a sequence in $Y_{u}$ converging to $A\in Y_{u}$. For simplicity, we denote $f_i^{A_{n}}$ and $f_i^{{A}}$ by $f_{n}$ and $f$ respectively. From the proof of  Prop. $II.13$ of \cite{hilgert&neeb}, $1_{A_{n}}(a)\to 1_{A}(a)$ pointwise for every $a \in \mathbb{R}^{d}\setminus\partial A$, where $\partial A$ is the  boundary of $A$.   Recall that 
\[
A=\{\sum_{i=1}^{d}x_{j}v_{j}\in \bbr^{d}: \textrm{  $x_{i}\leq f(\sum_{k\neq i}x_{k}v_{k})$}\}.\]

 Let $x\in Q_{i}$, and  $t>0$. Let $y=x+(f(x)-t)v_{i}$. Since $y\in Int(A)$, $y\in A_{n}$ eventually, i.e. for  $n\in\bbn$ sufficiently large, there exists $s_{n} \geq 0$ and $x_n \in Q_i$ such that  $y=x_{n}+(f_{n}(x)-s_{n})v_{i}$. Now,
\[f(x)-t=f_{n}(x)-s_{n}.\] for sufficiently large $n\in\mathbb{N}$. In particular, \[f_{n}(x)-f(x)\geq - t\] for large $n\in\mathbb{N}$. 

Let $z= x+(f(x)+t)v_{i}$. Then, $z\in \mathbb{R}^{d}\setminus A$. Therefore,  $z\in \mathbb{R}^{d}\setminus A_{n}$ eventually, i.e. for sufficiently large $n\in\mathbb{N}$, there exists $t_{n}>0$ and $z_n \in Q_i$  such that  $z=z_n+(f_{n}(x)+t_{n})v_{i}$. Now,
\[f(x)+t=f_{n}(x)+t_{n}\] for large $n\in\mathbb{N}$. Therefore, for large $n$, \begin{equation}
 f_{n}(x)-f(x)\leq t.
\end{equation}  Since $t>0$ is arbitrary,  we conclude that $f_{n}(x)\to f(x)$. This completes the proof.
\end{proof}

For $i\in\{1,2,\cdots , d\}$,  let $Y^{i}_{u}:=\{A\in Y_{u}: \textrm{$ f^{A}_{i}(0)=0$}\}$. Then, $Y^{i}_{u}$ is a closed subset of $Y_{u}$ by Lemma \ref{pointwise convergence}.
Define the map $\Psi_{i}:Y^{i}_{u}\times \mathbb{R} \to Y_u$ by
\[\Psi_i(A,t)= A+tv_{i}.\]
The map $\Psi_i$ is a homeomorphism with the inverse given by
$\Psi_{i}^{-1}(A)=(A-f^{A}_{i}(0)v_{i}, f_i^{A}(0))$. 
  Observe that \begin{equation}
      \label{equivariance of psi}
      \Psi_{i}(A,s+t)=\Psi_{i}(A,s)+tv_i
  \end{equation} for $A\in Y_{u}$, $s,t\in \bbr$. 
  
  \begin{remark}
  \label{full support}
  Let $i \in \{1,2,\cdots,d\}$, and let $X^{(i)}_{u+}:=\{A \in X_u: f_i^A(0)>0\}$. Then, $X^{(i)}_{u+}$ is open in $Y_u$ whose closure is $X_u$. 
  This is because $\Psi_i$ is a homeomorphism, $X_u=\Psi_i(Y_u^{(i)} \times [0,\infty))$ and $X^{(i)}_{u+}=\Psi_i(Y_u^{(i)}\times (0,\infty))$.
    \end{remark}

Let $\mu$ be an invariant, non-zero Radon measure on $Y_u$ which is fixed until further mention. The measure that we consider on $Y_u \times Y_u$ is the product measure $\mu \times \mu$. 
  \begin{lemma}\label{boundary null}  With the forgoing notation, \begin{enumerate}[(i)]\item for each $i\in\{1,2,\cdots ,d\}$, $Y^{i}_{u}\subset Y_{u}$ has measure zero,  and
\item the set $\mathcal{N}=\{(A,B): \textrm{$f^{A}_{i}(0)=f^{B}_{i}(0)$}\}\subset Y_{u}\times Y_{u}$ has measure zero.
\end{enumerate}
\end{lemma}
\begin{proof}Fix $i\in\{1,2,\cdots , d\}$.  Let $\nu=(\Psi_{i}^{-1})_{*}\mu$. Since $\mu$ is invariant, it follows from Eq. \ref{equivariance of psi} that $\nu$ is invariant under the action of $\mathbb{R}$ on $Y^{i}_{u} \times \bbr$ given by $s.(A,t)=(A,s+t)$. Hence, $\nu$ is a product measure of the form $\nu_0 \times m$ where $m$ is the Lebesgue measure on $\bbr$. Now, $\mu(Y^{i}_{u})=\nu(Y^{i}_{u}\times\{0\})=0$. This proves $(i)$.

Fix $B\in Y_{u}$, and let $t:=f^{B}_{i}(0)$. Then, $1_{\mathcal{N}}(A,B)=1$ if and only if $f^{A}_{i}(0)=f^{B}_{i}(0)=t$. But $\{A \in Y_u: \textrm{$f_i^{A}(0)=t$}\}=Y^{i}_{u}+tv_{i}$ is a null set. Now $(ii)$ follows from  Fubini's theorem. 
\end{proof}
 We denote the projection from $ \Gamma_{a}(L^{2}(X_{u}))$ onto its $n$-particle space by $P_{n}$.
For $A,B\in Y_{u}$, define
\[\epsilon_{i}(A,B):= \begin{cases} 1 & \textrm{if } f^{A}_{i}(0)>f^{B}_{i}(0) ,\\
-1 & \textrm{if } f^{A}_{i}(0)<f^{B}_{i}(0),\\
0 &\textrm{
if } f^{A}_{i}(0)=f^{B}_{i}(0).
\end{cases}\]
 Consider the additive cocycle  $\eta=\{\eta_{a}\}_{a\in P}$ for $ V^{\mu}$, where $\eta_{a}=1_{X_{u}\setminus X_{u}+a}$ for $a\in P$.
  Let $Exp_{i}$ denote the exponential map of the one parameter product system $\{E^{\mu}(tv_{i})\}_{t\geq 0}$.
\begin{lemma}\label{second particle} For $i\in \{1,2,\cdots ,d\}$ and $t>0$ ,
\[P_{2}\Big(Exp_{i}(\eta_{tv_{i}})\Big)(A,B)=\frac{1}{\sqrt{2}}\epsilon_{i}(A,B)\eta_{tv_{i}}(A) \eta_{tv_{i}} (B)\] for almost every $(A,B)\in X_{u}\times X_{u}$.
\end{lemma}
\begin{proof}
The computation is similar to the computation done in Prop. 5 of \cite{anbuCAR}. Fix $t>0$ and $i\in\{1,2,\cdots , d\}$. For simplicity, write $\eta_{sv_{i}}=\eta_{s}$ for $s\geq 0$. Let $x_{s}^{(1)}=\eta_{s}$. 

 By definition,
\begin{align}\label{secondparticle1}  P_{2}\Big(Exp_{i}(\eta_{tv_{i}})\Big)&=\int_{0}^{t}x_{s}^{(1)}d\eta_{s} \\
&=\lim_{n\to\infty} \sum_{j=0}^{n-1}\eta_{\frac{jt}{n}}.\eta_{\frac{t}{n}}\nonumber\\
&=\lim_{n\to\infty} \sum_{j=0}^{n-1}V_{\frac{jtv_{i}}{n}}\eta_{\frac{t}{n}}\wedge
\eta_{\frac{jt}{n}}\nonumber\\
&=\lim_{n\to\infty} \sum_{j=0}^{n-1}(\eta_{\frac{(j+1)t}{n}}-\eta_{\frac{jt}{n}})\wedge
\eta_{\frac{jt}{n}}\nonumber\\
\label{sec particle 1_1}&=\lim_{n\to\infty} \sum_{j=0}^{n-1}\eta_{\frac{(j+1)t}{n}}\wedge
\eta_{\frac{jt}{n}}.\end{align}
Let $s_{n}= \displaystyle \sum_{j=0}^{n-1}\eta_{\frac{(j+1)t}{n}}\wedge
\eta_{\frac{jt}{n}}$ for $n\in\bbn$. The proof will be over if we show that
\[\lim_{n\to\infty }s_{n}(A,B)=\frac{1}{\sqrt{2}}\epsilon_i(A,B)\eta_{tv_i}(A)\eta_{tv_i}(B)\] for almost every $(A,B)\in X_{u}\backslash(X_u+tv_i) \times X_u\backslash(X_u+tv_i)$. 

For $n\in\bbn$, suppose $A,B\in X_{u}$ are such that $f^{A}_{i}(0)<\frac{jt}{n}$, $\frac{jt}{n} \leq f^{B}_{i}(0)<  \frac{(j+1)t}{m}$ for some $j\in\{1,2,\cdots ,n-1\}$. Then, 
\begin{align*}s_{n}(A,B)
= 1_{X_{u}\setminus x_{u}+(\frac{(j+1)tv_{i}}{n})}\wedge 1_{X_{u}\setminus x_{u}+(\frac{jtv_{i}}{n})}(A,B).
\end{align*}
Since $A \in X_u\backslash (X_u+\frac{jtv_i}{n})$ and $B \in (X_u+\frac{jtv_i}{n})\backslash(X_u+\frac{(j+1)tv_i}{n})$, 
\begin{equation}
\label{value of sn}
s_n(A,B)=1_{X_{u}\setminus x_{u}+(\frac{(j+1)tv_{i}}{n})}\wedge 1_{X_{u}\setminus x_{u}+(\frac{jtv_{i}}{n})}(A,B)=-\frac{1}{\sqrt{2}}.
\end{equation}

Let $A,B\in X_{u}\setminus X_{u}+tv_{i}$. Suppose $f^{A}_{i}(0)<f^{B}_{i}(0)$.  For sufficiently large $n\in\mathbb{N}$, there exists a unique $j_{n}\in\{1,2,\cdots , n-1\}$  such that $f^{A}_{i}(0)< \frac{j_{n}t}{n}$, $\frac{j_{n}t}{n}\leq f^{B}_{i}(0)<  \frac{(j_{n}+1)t}{n}$. By Eq. \ref{value of sn}, for sufficiently large $n$, 
\begin{align*}\label{sec particle 2}s_n(A,B)&=-\frac{1}{\sqrt{2}} \\
&=\frac{1}{\sqrt{2}}\Big(\epsilon_{i}(A,B)\eta_{tv_{i}}(A) \eta_{tv_{i}}(B)\Big).\end{align*} 
Similarly, it may be proved that, when $f^A_i(0)>f_i^B(0)$, for $n$ sufficiently large, 
\[
s_n(A,B)=\frac{1}{\sqrt{2}}\epsilon_i(A,B)\eta_{tv_i}(A)\eta_{tv_i}(B). 
\]
The result follows  from Lemma \ref{boundary null}. The proof is complete. 
\end{proof}

\begin{lemma}\label{epsilon a=epsilon b} Suppose $\{Exp_{a}(1_{X_{u}\setminus X_{u}+a})\}_{a\in P}$ is a unit for $E^{\mu}$. Then, for every $i,j\in \{1,2,\cdots, d\}$, $\epsilon_{i}(A,B)=\epsilon_{j}(A,B)$ for almost every $(A,B)\in
X_{u}\times X_{u}$.
\end{lemma}
\begin{proof}   Write $V^{\mu}=V$, and let $\eta_{a}=1_{X_{u}\setminus (X_{u}+a)}$ for $a\in P$. Let $i,j \in \{1,2,\cdots,d\}$ be given. Since $Exp(\eta)$ is a unit,
\begin{align}\label{unitexp1}Exp_{i}(\eta_{sv_{i}}).Exp_{j}(\eta_{tv_{j}})=Exp_{j}(\eta_{tv_{j}}).Exp_{i}(\eta_{sv_{i}})\end{align} for $i,j\in\{1,2,\cdots , d\}$ and $s,t> 0$.

Hence,
\begin{align}\label{equality of sec particles}P_{2}\Big(Exp_{i}(\eta_{sv_{i}}).Exp_{j}(\eta_{tv_{j}})\Big)=P_{2}\Big(Exp_{j}(\eta_{tv_{j}}).Exp_{i}(\eta_{sv_{i}})\Big)\end{align}for $i,j\in\{1,2,\cdots , d\}$ and $s,t> 0$.

Thanks to Lemma \ref{second particle}, 
\begin{align*}P_{2}\Big(Exp_{i}(\eta_{sv_{i}}).Exp_{j}(\eta_{tv_{j}})\Big)(A,B)&=\epsilon_i(A,B)
\eta_{sv_{i}}(A) \eta_{sv_{i}}(B)+V_{sv_{i}}\eta_{tv_{j}}\wedge \eta_{sv_{i}}(A,B)\\&+\epsilon_{j}(A-sv_i,B-sv_i)V_{sv_{i}}\eta_{tv_{j}}(A) V_{sv_{i}}\eta_{tv_{j}}(B)\end{align*} for almost every $(A,B)\in X_{u}\times X_{u}$.  Similarly, 
\begin{align*}P_{2}\Big(Exp_{j}(\eta_{tv_{j}}).Exp_{i}(\eta_{sv_{i}})\Big)(A,B)&=\epsilon_j(A,B)
\eta_{tv_{j}}(A) \eta_{tv_{j}}(B)+V_{tv_{j}}\eta_{sv_{i}}\wedge \eta_{tv_{j}}(A,B)\\&+\epsilon_{i}(A-tv_j,B-tv_j)V_{tv_{j}}\eta_{sv_{i}}(A) V_{tv_{j}}\eta_{sv_{i}}(B)\end{align*} for almost every $(A,B)\in X_{u}\times X_{u}$.

For $s>0$ and $i\in\{1,2,\cdots , d\}$, let $L_{i,s}:=(1_{X_{u}}\setminus X_{u}+sv_{i})\times (1_{X_{u}}\setminus X_{u}+sv_{i})$. Let $s,t >0$. For almost all $(A,B)\in L_{i,s}\cap L_{j,t}$, \[P_{2}\Big(Exp_{i}(\eta_{sv_{i}}).Exp_{j}(\eta_{tv_{j}})\Big)=\epsilon_{i}(A,B)\] and 
\[P_{2}\Big(Exp_{j}(\eta_{tv_{j}}).Exp_{i}(\eta_{sv_{i}})\Big)=\epsilon_{j}(A,B).\]
Therefore,  Eq. \ref{equality of sec particles} implies,  
\[\epsilon_{i}(A,B)=\epsilon_{j}(A,B)\] for almost every $(A,B)\in L_{i,s}\cap L_{j,t}$. Now observe that $X_{u}\times X_{u}=\displaystyle{\bigcup_{m,n\in\mathbb{N}}}L_{i,m}\cap L_{j,n}$. Hence the proof follows.
\end{proof}
\textbf{Notation: }  Let  $X:=supp(\mu) \cap X_u$. Recall that $Y_u\backslash supp(\mu)$ is the largest open set that has $\mu$ measure zero. 
Since $\mu$ is invariant, $supp(\mu)$ is $\bbr^d$-invariant. It follows from Remark \ref{full support} that $\mu|_{X}$ has full support. 

\begin{prop}\label{total order} Suppose $\{Exp_{a}(1_{X_{u}\setminus X_{u}+a})\}_{a\in P}$ is a unit for $E^{\mu}$. Then, for  $A,B\in supp(\mu)$, $A\subset B$ or $B\subset A$.
\end{prop}
\begin{proof} Write $\eta_{a}=1_{X_{u}\setminus X_{u}+a}$ for $a\in P$. Assume that $Exp(\eta)$ is a unit for $E^{\mu}$. For $A,B\in Y_{u}$ and $i\in\{1,2,\cdots ,d\}$, note that $A\subset B$ if and only if $f_{i}^{A}(x)\leq f^{B}_{i}(x)$ for each $x\in Q_{i}$.

  Let  $A,B\in X$, and $i\in\{1,2,\cdots ,d\}$. Suppose $f^{A}_{i}(0)\leq f^{B}_{i}(0)$. Then, we claim the following.
\begin{enumerate}[(i)]
\item $f^{A}_{j}(0)\leq f^{B}_{j}(0)$ for each $j\in\{1,2,\cdots ,d\}$, and
\item if $t_{j}>0$ for $j\in\{1,2,\cdots ,d-1\}$, then \[f^{A}_{d}(-\sum_{j=1}^{d-1}t_{j}v_{j})\leq f^{B}_{d}(-\sum_{j=1}^{d-1}t_{j}v_{j}).\]
\end{enumerate}
Observe that $X$ is a measurable subset of $X_{u}$ such that $X+P \subset X$. Consider the set \[N_{i,j}=\{(A^{'},B^{'})\in X\times X: \textrm{$f^{A^{'}}_{i}(0)<f^{B^{'}}_{i}(0)$, $f^{A^{'}}_{j}(0)>f^{B^{'}}_{j}(0)$}\}\] for $i,j\in\{1,2,\cdots , d\}$. By Prop. \ref{pointwise convergence}, $N_{i,j}$ is an open set, and  by Lemma \ref{epsilon a=epsilon b}, $N_{i,j}$ is a null set. Therefore, $N_{i,j}$ is empty.  Similarly, consider the set \[ M_{i,j}=\{(A^{'},B^{'}) \in X \times X: \textrm{ $f^{A^{'}}_{i}(0)=f^{B^{'}}_{i}(0)$, $f^{A^{'}}_{j}(0)>f^{B^{'}}_{j}(0)$}\}\] for $i,j\in\{1,2,\cdots ,d\}$. Fix $i,j\in \{1,2,\cdots ,d\}$ with $i \neq j$.  Suppose $(A^{'},B^{'})\in M_{i,j}$. Let $s,t>0$ be sufficiently small such that $s<t$ and $f^{A^{'}}_{j}(0)>f^{B^{'}}_{j}(-tv_{i})$. Then,  \[
f_i^{A^{'}+sv_i}=f^{A^{'}}_{i}(0)+s<f^{B^{'}}_{i}(0)+t=f_i^{B^{'}+sv_i}.\] 
Since $f^{A^{'}}_{j}$ is decreasing,\[f_{j}^{A^{'}+sv_i}(0)=f^{A^{'}}_{j}(-sv_{i})\geq f^{A^{'}}_{j}(0)>f^{B^{'}}_{j}(-tv_{i})=f_j^{B^{'}+tv_i}(0).\] 
This implies $(A^{'}+sv_{i},B^{'}+tv_{i})\in N_{i,j}$, which is a contradiction. Hence, $M_{i,j}=\emptyset$. Therefore, for $A^{'},B^{'}\in X$ and $i,j\in\{1,2,\cdots d\}$, $f^{A^{'}}_{i}(0)\leq f^{B^{'}}_{i}(0)$ if and only if   $f^{A^{'}}_{j}(0)\leq f^{B^{'}}_{j}(0)$, which proves $(i)$.

 For each $j \in\{1,2,\cdots ,d-1\}$, let $t_{j}>0$.   Let $A,B \in X$ be such that $f_i^A(0) \leq f_i^B(0)$ for some $i$. 
Since $f_{i}^{A}(0)\leq  f_{i}^{B}(0)$, by $(i)$, $f_{1}^{A}(0)\leq f_{1}^{B}(0)$. Since $f^{A+t_{1}v_{1}}_{1}(0)\leq f^{B+t_{1}v_{1}}_{1}(0)$, by $(i)$,  $f^{A+t_{1}v_{1}}_{2}(0)\leq f^{B+t_{1}v_{1}}_{2}(0)$, i.e.
 \[f^{A}_{2}(-t_{1}v_{1})\leq f^{B}_{2}(-t_{1}v_{1}).\]
   
Since $f_{2}^{A+t_{1}v_{1}}(0)+t_{2}\leq f^{B+t_{1}v_{1}}_{2}(0)+t_{2}$,
\[f^{A+t_{1}v_{1}+t_{2}v_{2}}_{2}(0)\leq f^{B+t_{1}v_{1}+t_{2}v_{2}}_{2}(0).\] Once again by $(i)$, 
 \[f^{A}_{3}(-t_{1}v_{1}-t_{2}v_{2})=f_3^{A+t_1v_1+t_2v_2}(0)\leq f_3^{B+t_1v_1+t_2v_2}(0) =  f^{B}_{3}(-t_{1}v_{1}-t_{2}v_{2}).\]

Inductively,
\begin{align}\label{negative}f^{A}_{d}(-\sum_{j=1}^{d-1}t_{j}v_{j})\leq f^{B}_{d}(-\sum_{j=1}^{d-1}t_{j}v_{j}).\end{align} We have proved $(ii)$.

It follows from $(i)$ that for $i \in \{1,2,\cdots,d\}$ and $A,B \in X$, $f_i^A(0)>f_i^B(0)$ if and only if $f_j^A(0)>f_j^B(0)$ for every $j \in \{1,2,\cdots,d\}$. Suppose $f_i^A(0)>f_i^B(0)$ for $A,B \in X$ and $i \in \{1,2,\cdots,d\}$. Arguing as before, we see that \begin{equation}
    \label{too much positivity}
    f^{A}_{d}(-\sum_{j=1}^{d-1}t_{j}v_{j})>f^{B}_{d}(-\sum_{j=1}^{d-1}t_{j}v_{j})
\end{equation} whenever $t_{j}>0$ for $j\in\{1,2,\cdots ,d-1\}$.

 Let $A,B\in X$. Without loss of generality, we can assume  that $f^{A}_{d}(0)\leq f^{B}_{d}(0)$. Then, 
 \begin{equation}
     \label{negative1}
         f^{A}_{d}(-\sum_{j=1}^{d-1}t_{j}v_{j}) \leq f^{B}_{d}(-\sum_{j=1}^{d-1}t_{j}v_{j})
\end{equation} whenever $t_{j}>0$ for $j\in\{1,2,\cdots ,d-1\}$.
 
   Let $x \in Q_d$. Choose $t>0$ such that $A-x+tv_{d}\in X$ and $B-x+tv_{d}\in X$,  i.e.
 \[f_{d}^{A-x+tv_{d}}(0)=f_d^{A}(x)+t\geq 0\] and \[f_{d}^{B-x+tv_{d}}(0)=f_d^B(x)+t\geq 0.\]
  Let $\widetilde{A}=A-x+tv_{d}$ and $\widetilde{B}=B-x+tv_{d}$.
 Suppose $f^{A}_{d}(x)>f^{B}_{d}(x)$. Then,
 \[f^{\widetilde{A}}_{d}(0)=f^{A}_{d}(x)+t> f^{B}_{d}(x)+t=f^{\widetilde{B}}_{d}(0).\] 
By Eq. \ref{too much positivity},\[f^{\widetilde{A}}_{d}(-\sum_{j=1}^{d-1}r_{j}v_{j})>f^{\widetilde{B}}_{d}(-\sum_{j=1}^{d-1}r_{j}v_{j})\] whenever $r_{j}>0$ for $j=1,2,\cdots , d-1$. This implies
\[f^{A}_{d}(\sum_{j=1}^{d-1}(x_{j}-r_{j})v_{j})>f^{B}_{d}(\sum_{j=1}^{d-1}(x_{j}-r_{j})v_{j})\] whenever $r_{j}>0$ for $j=1,2,\cdots , d-1$. This contradicts Eq. \ref{negative1}.
 Hence, $f^{A}_{d}(x)\leq f^{B}_{d}(x)$ for every $x\in Q_{d}$, i.e $A\subset B$.

 Now suppose $A,B\in supp(\mu)$. Then,  there exists $t>0$ such that $A+tv_{1}, B+tv_{1}\in X$. Without loss of generality, assume that $A+tv_{1}\subset B+tv_{1}$. Then,  $A\subset B$. Therefore, if $A,B\in supp(\mu)$, $A\subset B$ or $B\subset A$.  The proof is complete.
\end{proof}

\begin{theorem} 
\label{the most crucial theorem}The following statements are equivalent.
\begin{enumerate}[(i)]
\item $\{Exp(1_{X_{u}\setminus X_{u}+a})\}_{a\in P}$ is a unit for $E^{\mu}$.
    \item There exists  $\lambda \in S(P^{*})$ such that $supp(\mu)=Y^{\lambda}_{u}$.
\end{enumerate}
\end{theorem}

\begin{proof} By Lemma \ref{<-}, if $supp(\mu)=Y^{\lambda}_{u}$ for some $\lambda\in S(P^{*})$, then $\{Exp(1_{X_{u}\setminus X_{u}+a})\}_{a\in P}$ is a unit for $E^{\mu}$.

To prove the converse, let $A\in supp(\mu)$. By Lemma \ref{total order}, for $x\in\mathbb{R}^{d}$, either $A+x\subset A$ or $A\subset A+x$.
Define $Q_{A}:=\{x\in\mathbb{R}^{d}: \textrm{$A+x\subset A$}\}$. Then, $Q_{A}\cup -Q_{A}=\mathbb{R}^{d}$. Note that $-P\subset Q_{A}$ and $-P+Q_A \subset Q_A$. 
Let $Q:=Q_{A}\cap -Q_{A}$. We claim the following:
\begin{enumerate}[(i)]
\item $\partial Q_{A}=Q$, where $\partial Q_A$ is the boundary of $Q_{A}$,  and
\item $Q$ is a vector subspace of $\mathbb{R}^{d}$.
\end{enumerate} We identify $span\{v_{1},v_{2},\cdots v_{d-1}\}$ with $\mathbb{R}^{d-1}$. Let $f:\mathbb{R}^{d-1}\to\mathbb{R}$ be a continuous decreasing function such that $A=\{(x,t): \textrm{$t\leq f(x)$}\}$. Note that
\[Q_{A}
=\{(y,t) \in \bbr^{d-1}\times \bbr: \textrm{$f(x)\leq f(x+y)-t$ for all $x\in\mathbb{R}^{d-1}$ }\},\]
 and \[-Q_{A}=\{(y,t)\in\bbr^{d-1}\times\bbr: \textrm{$f(x)\geq f(x+y)-t$ for all $x\in\mathbb{R}^{d-1}$ }\}.\]
It is not difficult to deduce using the fact that $Q_A \cup -Q_A=\bbr^d$ that
\[\partial Q_{A}\subset \{(y,t)\in\bbr^{d-1}\times\bbr: \textrm{$f(x)= f(x+y)-t$ for all $x\in\mathbb{R}^{d-1}$ }\}=Q.\]

Suppose $(y,t) \in Q$. Set $t_n:=t-\frac{1}{n}$ and $s_n:=t+\frac{1}{n}$. Then, $(y,t_n) \in Q_A$ and $(y,t_n) \to (y,t)$. Also, $(y,s_n) \notin Q_A$ and $(y,s_n) \to (y,t)$. Therefore, $(y,t) \in \partial Q_A$. Hence, $\partial Q_A=Q$.

Let $g:\bbr^{d-1}\to\bbr$ be a continuous decreasing function such that $Q_{A}=\{(x,t): \textrm{$t\leq g(x)$}\}$. 
Clearly, $Q$ is a closed subgroup of $\mathbb{R}^{d}$. Since $Q=\partial Q_A$ is the graph of the continuous function $g$, it is connected. Therefore, $Q$ is a vector space and consequently, $g$ is linear. Hence, there exists $\lambda \in \bbr^d$, $||\lambda||=1$ such that 
\[Q_{A}=\{y\in\mathbb{R}^{d}: \textrm{$\langle\lambda|y\rangle\leq 0$}\}.\]
  Since $-P\subset Q_{A}$, $\lambda\in S(P^{*})$.

 Since $Q_{A}+A\subset A$, there exists $y\in\mathbb{R}^{d}$ such that $A=y+Q_{A}$, i.e. $A \in Y_u^{\lambda}$. Since $supp(\mu) \subset Y_{u}$ is invariant, $x+A=x+y+Q_A\in  supp(\mu)$ for each $x\in\mathbb{R}^{d}$, i.e. $Y^{\lambda}_{u}\subset supp(\mu)$. To stress the dependence of $\lambda$ on $A$, we write $\lambda=\lambda_{A}$. We have proved that  \[supp(\mu)=\displaystyle \bigcup_{A\in supp(\mu)}Y^{\lambda_{A}}_{u}.\]

Suppose $A, B\in supp(\mu)$. Note that $Q_{A}\in Y^{\lambda_{A}}_{u}\subset supp(\mu)$ and $Q_{B}\in Y^{\lambda_{B}}_{u}\subset supp(\mu)$. Then, by Prop. \ref{total order}, either $Q_{A}\subset Q_{B}$ or $Q_{B}\subset Q_{A}$. But this can happen only if $\lambda_{A}=\lambda_{B}$. In that case, $Y^{\lambda_{A}}_{u}=Y^{\lambda_{B}}_{u}$. Therefore, there exists $\lambda\in S(P^{*})$ such that $supp(\mu)=Y^{\lambda}_{u}$. Now the proof is complete.
\end{proof}

\textbf{Notation and Convention: } Let $V$ be an isometric representation of $P$ on a Hilbert space $\clh$. Let $\xi=\{\xi_{a}\}_{a\in P}\in \mathcal{A}(V)$.  We define $\mathcal{H}^{\xi}$ to be the smallest, closed, reducing subspace of $\mathcal{H}$ containing $\{\xi_{a}:\textrm{ $a\in P$}\}$. (Hereafter, all reducing subspaces will be assumed to closed.) 
If $W$ is a direct summand of $V$, we view the product system $E^W$ as a subsystem of $E^V$.  Similarly, we view $\mathcal{A}(W)$ as a subspace of $\mathcal{A}(V)$.

\begin{prop}\label{subspace}
Let $V$ be a pure isometric representation  of $P$ on a Hilbert space $\mathcal{H}$ with commuting range projections.
Let $\xi=\{\xi_{a}\}_{a\in P}\in\mathcal{A}(V)$ be non-zero. Then,  \begin{enumerate}\item there exists an $\bbr^d$-invariant, non-zero Radon measure $\mu$ on $Y_u$ such that $V|_{\mathcal{H}^{\xi}}$ is unitarily equivalent to $V^{\mu}$, and
\item if $Exp(\xi)$ is a unit for $E^{V}$, then there exists $\lambda\in S(P^{*})$ such that $V|_{\mathcal{H}^{\xi}}$ is unitarily equivalent to $S^{\lambda}$.\end{enumerate}
\end{prop}
\begin{proof}For the proof of $(1)$, we refer the reader to Theorm 3.16 and Remark 3.17 of \cite{sundarkms}. From $(1)$, $V^{\xi}:=V|_{\mathcal{H}^{\xi}}$ is unitarily equivalent to $V^{\mu}$. From the proof of Theorem 3.16 of \cite{sundarkms} we can see that the unitary $U$ intertwining $V^{\xi}$ and $V^{\mu}$ can be chosen such that $U\xi_{a}=1_{X_{u}\setminus X_{u}+a}$ for $a\in P$. Suppose $Exp(\xi)$ is a unit for $E^{V}$. Note that  the subsystem $E^{V^\xi}$ contains $Exp(\xi)$. Hence, $\{Exp(1_{X_{u}\setminus X_{u}+a})\}_{a\in P}$ is a unit for $E^{V^\mu}$. It follows from Thm. \ref{the most crucial theorem} and  Lemma \ref{<-} that $V^{\mu}$ is unitarily equivalent to $S^{\lambda}$ for some $\lambda\in S(P^{*})$. This proves $(2)$.
\end{proof}
 
 Before proceeding further, let us recall the notation introduced in the introduction. 
 For $\lambda\in S(P^{*})$ and $k\in\mathbb{N}_\infty$, let $S^{(\lambda, k)}$ denote the isometric representation $\{S_{\langle\lambda|a\rangle}\otimes 1\}_{a\in P}$ of $P$ acting on the Hilbert space $L^{2}[0,\infty)\otimes \clk$ where $\clk$ is a  Hilbert space of dimension $k$. For $k=1$, we denote $S^{(\lambda,1)}$ by $S^\lambda$.
 For a non-empty countable  set $I$, an injective map $\lambda:I \to S(P^*)$ and a function $k:I \to \bbn_\infty$, set 
\[
S^{(\lambda,k)}:=\bigoplus_{i \in I}S^{(\lambda_i,k_{i})}.
\]
 Let $E^{(\lambda,k)}$ be the product system of  the CAR flow associated with $S^{(\lambda,k)}$.

 \begin{remark}
 \label{disjointness}
 A few properties concerning the representation $S^{(\lambda,k)}$ are summarised below.
\begin{enumerate}
 \item[(1)] For $\lambda_{1},\lambda_{2}\in S(P^*)$ and for $k_1,k_2 \in \bbn_\infty$, $S^{(\lambda_{1}, k_{1})}$ is unitarily equivalent to $S^{(\lambda_{2}, k_{2})}$ if and only if $\lambda_{1}=\lambda_{2}$ and $k_{1}=k_{2}$.  
  Also, $S^{\lambda}$ is irreducible for every $\lambda\in S(P^{*})$, i.e. it  has no non-zero non-trivial  reducing subspace.
Thus, if $\lambda_1 \neq \lambda_2$, the representations $S^{(\lambda_1,k_1)}$ and $S^{(\lambda_2,k_2)}$ are disjoint.
Recall that two isometric representations $V$ and $W$, acting on $\clh$ and $\clk$ respectively,  are said to be disjoint if 
\[
\{T \in B(\clh,\clk): TV_a=W_aT, TV_a^*=W_a^*T \textrm{~~for all $a \in P$}\}=0.
\]

\item[(2)] Consider a countable(non-empty) indexing set $I$. Let $\lambda: I\to S(P^{*})$ be injective, and let $k:I\to \bbn_{\infty}$ be a map. Consider $V=S^{(\lambda,k)}=\bigoplus_{i \in I}S^{(\lambda_i,k_{i})}$ acting on the Hilbert space $\clh=\displaystyle{\bigoplus_{i\in I}}L^{2}[0,\infty)\otimes \mathcal{K}_{i}$, where $\clk_{i}$ is of dimension $k_{i}$. Write $V^{(i)}= S^{(\lambda_{i}, k_{i})}$.
Then, $V^{(i)}$ acts on $\clh_i:=L^{2}[0,\infty)\otimes \clk_i$.  Since $\lambda_{i}\neq \lambda_{j}$ whenever $i\neq j$, the isometric representations $V^{(i)}$ and $V^{(j)}$ are disjoint whenever $i \neq j$. Thus, a bounded operator $T \in \{V_{a}, V^{*}_{a}: \textrm{ $a\in P$}\}^{'}$, if and only if there exists $T_{i}\in \{V^{(i)}_{a}, V^{(i)*}_{a}: \textrm{ $a\in P$}\}^{'}$ such that $T|_{\clh_{i}}=T_{i}$ for each $i\in I$. Morevoer, $\{V^{(i)}_a, V^{(i)*}_a:a \in P\}^{'}=\{1 \otimes R: R \in B(\clk_i)\}$.
Hence, 
\[
\{V_a,V_a^*:a \in P\}^{'}=\bigoplus_{i \in I}B(\clk_i).
\]
\item[(3)] It follows from $(2)$ that   the reducing subspaces of $V$ are of the form $\displaystyle \bigoplus_{j\in J}L^{2}[0,\infty)\otimes \mathcal{W}_{j}$ for some non-empty subset $J$ of $I$, and where, for $j$, $\mathcal{W}_j$ is a  subspace of $\clk_j$. Suppose $\mathcal{W}$ is a non-zero reducing subspace for $V$ such that $V|_{\mathcal{W}}$ is irreducible. Then, there exists $i \in I$ and a one dimensional subspace $\mathcal{W}_i$ of $\clk_i$ such  $\mathcal{W} \subset \clh_i=L^{2}[0,\infty)\otimes \clk_{i}$  and $\mathcal{W}=L^2[0,\infty) \otimes \mathcal{W}_i$. Moroever,  $V|_{\mathcal{W}}$ is unitarily equivalent to $S^{\lambda_{i}}$.

 \end{enumerate}

\end{remark}

The next proposition is the ``only if part" of Thm. \ref{main intro}. The ``uniqueness" part of Thm. \ref{main intro} follows from   the fact that if an isometric representation admits a direct sum decomposition of irreducible representations, then the decomposition is ``unique".

\begin{prop}
Let $V$ be a pure isometric representation of $P$ with commuting range projections on a Hilbert space $\clh$. Suppose the product system $E^{V}$ of the  CAR flow associated with $V$ is type I. Then,  there exists a non-empty, countable set 
 $I$, a map $\lambda:I\to S(P^{*})$ which is injective, and a map $k:I\to \bbn_{\infty}$ such that $V$ is unitarily equivalent to $S^{(\lambda,k)}$. Equivalently, $E^V$ is isomorphic to $E^{(\lambda,k)}$.
\end{prop}
\begin{proof} A family $\mathscr{S}$ of closed  subspaces of $\mathcal{H}$ is said to be ``shift reducing" if  \begin{enumerate}[(1)]\item each $\mathcal{K}\in\mathscr{S}$ is a non-zero reducing subspace for $V$,
 \item for each $\mathcal{K}\in \mathscr{S}$, there exists $\lambda\in S(P^*)$ such that $V|_{\mathcal{K}}$ is unitarily equivalent to $S^{\lambda}$, and
 \item if $\mathcal{K},\mathcal{L}\in\mathscr{S}$ and $\mathcal{K}\neq \mathcal{L}$, then $\mathcal{K}\perp \mathcal{L}$.
 \end{enumerate} 
 Let $\mathcal{W}:=\{\mathscr{S}: \textrm{$\mathscr{S}$  is shift reducing}\}$. Note that 
$\mathcal{W}$ is partially ordered where the partial order on $\mathcal{W}$ is given by inclusion. 

 Since $E^{V}$ is type $I$, there exists a non-zero additive cocycle $\xi=\{\xi_{a}\}_{a\in P}$ such that $Exp(\xi)$ is a unit for $E^{V}$. Hence, $\{\mathcal{H}^{\xi}\}$ is in $\mathcal{W}$ by
Prop. \ref{subspace}. Therefore, $\mathcal{W}$ is non-empty. 
A routine application of Zorn's lemma allows us to get a maximal shift reducing family $\mathscr{K}$ of closed subspaces.
  
Define 
\[\mathcal{K}:=\displaystyle{\bigoplus_{W\in\mathscr{K}}}W.\] Note that for each $W\in \mathcal{K}$, there exists $\lambda_{W}\in S(P^{*})$ such that $V|_{W}$ is unitarily equivalent to $S^{\lambda_{W}}$. Hence, $V|_{\mathcal{K}}=\displaystyle{\bigoplus_{W\in\mathcal{K}}}S^{\lambda_{W}}$. Let $\mathcal{L}:=\mathcal{K}^{\perp}$. Note that $\mathcal{K}$ and  $\mathcal{L}$ are reducing subspaces for $V$. Denote the restriction of $V$ to $\mathcal{K}$ and $\mathcal{L}$ by $V^{(1)}$ and $V^{(2)}$ respectively. Also, denote the projection of $\mathcal{H}$ onto $\mathcal{L}$ by $Q$. 

It suffices to prove that $V=V^{(1)}$. 
Denote by $E^{V^{(1)}}$ the product system of the CAR flow associated with $V^{(1)}$. We consider $E^{V^{(1)}}$ as a subsystem of $E^{V}$.  We claim that $E^{V^{(1)}}=E^{V}$. 
 Suppose $u=\{u_{a}\}_{a\in P}$ is a unit for $E^{V}$. We claim that $E^{V^{(1
)}}$ contains $u$. Without loss of generality, we can assume $u$ is exponential.  Then, there exists $\xi=\{\xi_{a}\}_{a\in P}\in \mathcal{A}(V)$ such that $u_{a}=Exp_{a}(\xi_{a})$ for $a\in P$. For $a\in P$, let $\xi^{(1)}_{a}=(1-Q)\xi_{a}$  and $\xi^{(2)}_{a}=Q\xi_{a}$. Then, $\xi^{(i)}=\{\xi^{(i)}_{a}\}_{a\in P}\in\mathcal{A}(V^{(i)})$ for $i=1,2$.

 Note that $\mathcal{Q}=\{\mathcal{Q}_{a}\}_{a\in P}:E^{V}\to E^{V}$ given by
 \[E^{V}(a)\ni u\mapsto \Gamma(Q)u\in E^{V}(a)\] is multiplicative, where $\Gamma(Q)$ is the second quantisation of $Q$. Also, $\Gamma(Q)u_a=Exp_a(\xi^{(2)}_a)$. 
 Therefore, $\{Exp_{a}(\xi^{(2)}_{a})\}_{a\in P}$ is a unit for $E^{V}$.
 
 Assume that $\xi^{(2)}$ is non-zero. Consider the reducing subspace $\mathcal{H}^{\xi^{(2)}}$ of $V$. Then, by Prop. \ref{subspace} there exists $\lambda\in S(P^{*})$ such that  $V|_{\mathcal{H}^{\xi^{(2)}}}$ is unitarily equivalent to $S^{\lambda}$. Now, $\mathscr{K}\cup \{\mathcal{H}^{\xi^{(2)}}\}\in\mathcal{W}$ and  contains $\mathscr{K}$ as a proper subset. This contradicts the maximality of $\mathscr{K}$. Therefore, $\xi^{(2)}=0$. Thus, $E^{V^{(1)}}$ contains $u$. Hence, every  unit of $E^{V}$ is contained in $E^{V^{(1)}}$. Since $E^{V}$ is type I, $E^{V}=E^{V^{(1)}}$. 
 
 Since \[E^{V^{(1)}}(a)=\Gamma_{a}(Ker(V^{(1)*}_a))=\Gamma_{a}(Ker(V_{a}^{*}))=E^V(a)\] for $a\in P$, $Ker(V^{(1)*}_{a})=Ker(V_{a}^{*})$ for $a\in P$. Since $\displaystyle{\bigcup_{a\in P}}Ker(V_{a}^{*})$ is dense in $\clh$, $\mathcal{H}=\mathcal{K}$.  Therefore, $V=V^{(1)}$. The proof is complete. 
\end{proof}

Fix a countable (non-empty) indexing set $I$. Let $\lambda: I\to S(P^{*})$ be injective, and let $k:I\to \bbn_{\infty}$ be a map. Let $V=S^{(\lambda,k)}$. Denote the space $\displaystyle{\bigoplus_{i\in I}}L^{2}[0,\infty)\otimes \mathcal{K}_{i}$ (on which $V$ acts) by $\clh$, where $\mathcal{K}_{i}$ is a Hilbert space of dimension $k_{i}\in \bbn_{\infty}$. For each $i\in I$, let $V^{(i)}=S^{(\lambda_{i}, k_{i})}$, and write $\clh_{i}:=L^{2}[0,\infty)\otimes \mathcal{K}_{i}$ (the Hilbert space on which $V^{(i)}$ acts).
Note that  for $i\in I$, we may view $\mathcal{A}(S^{(\lambda_i,k_{i})})$ as a subspace of $\mathcal{A}(V)$ under the natural inclusion.
\begin{prop}
\label{additive cocycles that give units} 
Let $\xi \in \mathcal{A}(V)$. Then, $\{Exp_{a}(\xi_{a})\}_{a\in P}$ is a unit of $E^V$ if and only if there exists $i\in I$ such that $\xi\in \mathcal{A}(S^{(\lambda_i,k_{i})})$.
\end{prop}
\begin{proof} Suppose $\xi=\{\xi_{a}\}_{a\in P} \in \mathcal{A}(V)$ is non-zero and is such that $Exp(\xi)$ is a unit of $E^{V}$. Consider the reducing subspace $\mathcal{H}^{\xi}$ of $V$. By Prop. \ref{subspace}, $V|_{\mathcal{H}^{\xi}}$ is unitarily equivalent to $S^{\lambda}$ for some $\lambda\in S(P^{*})$.   Since $S^{\lambda}$ is irreducible, by Remark \ref{disjointness}, $\mathcal{H}^{\xi}$ is a subspace of $\mathcal{H}_{i}$ and $\lambda=\lambda_{i}$ for some $i\in I$. In particular, $\xi\in\mathcal{A}(S^{\lambda_{i}, k_{i})})$.

Conversely, suppose $\xi=\{\xi_{a}\}_{a\in P}\in\mathcal{A}(S^{(\lambda_{i},k_{i})})$ for some $i\in I$. Then, there exists an additive cocycle $\eta=\{\eta_{a}\}\in \mathcal{A}(S^{\lambda_{i}})$ and  $\gamma\in \mathcal{K}_{i}$ such that $\xi_{a}=\eta_{a}\otimes \gamma$ for $a\in P$. In turn, there is an additive cocycle $\widetilde{\eta}=\{\widetilde{\eta}_{t}\}_{t\geq 0}$ of $\{S_{t}\}_{t\geq 0}$ such that $\eta_{a}=\widetilde{\eta}_{\langle\lambda_{i}|a\rangle}$ for $a\in P$. Denote the exponential map of the product system of the $1$-parameter CAR flow  associated with $\{S_{t}\otimes 1\}_{t\geq 0}$ acting on $\clh_{i}$ by $Exp$. By  Remark \ref{rmk},
 \[Exp(\widetilde{\eta}_{s}\otimes\gamma). Exp(\widetilde{\eta}_{t}\otimes \gamma)=Exp(\widetilde{\eta}_{s+t}\otimes \gamma)\] for $s,t\geq 0$.
Then, for $a,b \in P$,
\begin{align*}Exp_{a}(\xi_{a}).Exp_{b}(\xi_{b})=&Exp_{a}(\eta_{a}\otimes\gamma).Exp_{b}(\eta_{b}\otimes \gamma)\\
=& Exp(\widetilde{\eta}_{\langle\lambda_{i}|a\rangle}\otimes\gamma).Exp(\widetilde{\eta}_{\langle\lambda_{i}|b\rangle}\otimes\gamma)\\
=&Exp(\widetilde{\eta}_{\langle\lambda_{i}|a+b\rangle}\otimes\gamma)\\
=& Exp_{a+b}(\xi_{a+b})
\end{align*} for $a,b\in P$, i.e. $Exp(\xi)$ is a unit. This completes the proof.
\end{proof}

The following is the `if part' of Thm. \ref{main intro}.

\begin{prop} Keeping the forgoing notation,  the CAR flow associated with $V$ is type I.
\end{prop}
\begin{proof} 
Let $F$ be a subsystem of $E^{V}$ containing all the units of $E^{V}$. For $a\in P$, let $\Psi_{a}$ be the projection of $E^{V}(a)$ onto $F(a)$. Then, $\Psi=\{\Psi_{a}\}_{a\in P}$ is multiplicative, i.e. \[\Psi_{a}.\Psi_{b}=\Psi_{a+b}\] for $a,b\in P$.
Let $a\in P$. Consider the one parameter product system $E_{a}=\{E(ta)\}_{t\geq 0}$ and  the multiplicative section of maps $\{\Psi_{ta}\}_{t\geq 0}$. Let $\widetilde{E}_{a}$ be the product system of the $1$-parameter CCR flow associated with $\{V_{ta}\}_{t\geq 0}$.   Suppose $\widetilde{\Psi}=\{\mathcal{Q}_{t}\}_{t\geq 0}:\widetilde{E}_{a}\to \widetilde{E}_{a}$ is such that \begin{enumerate}[(i)]\item $\mathcal{Q}_{t}$ is a projection for $t\geq 0$, and
\item $\mathcal{Q}_{s}.\mathcal{Q}_{t}=\mathcal{Q}_{s+t}$ for $s,t\geq 0$.
\end{enumerate}
Then, by Thm. 7.6 of \cite{BhatCocycles} (see also Prop. 6.12 of \cite{SUNDAR}),\ there exists an additive cocycle $\{\xi_{t}^{a}\}_{t\geq 0}$ of $\{V_{ta}\}_{t\geq 0}$, a projection $Q^{a} \in\{V_{ta},V_{ta}^{*}: \textrm{ $t\geq 0$}\}^{'}$ such that $(1-Q^{a})\xi^{a}_{t}=\xi^{a}_{t}$ for $t\geq 0$, and $\mu_{a}\in\mathbb{R}$ such that
\[\widetilde{\Psi}_{t}e(\eta)=e^{\mu_{a} t}e^{\langle\eta|\xi^{a}_{t}\rangle} e(Q^{a}\eta+\xi^{a}_{t})\] for $\eta\in Ker(V_{ta}^{*})$ and $t\geq 0$.

By Lemma \ref{isoccrcar}, the map $\widetilde{E}_{a}(t)\ni exp(\xi)\to Exp_{a}(\xi)\in E_{a}(t)$ extends to an isomorphism from $\widetilde{E}_{a}$ onto $E_{a}$. Therefore, there exists an additive cocycle $\{\xi_{t}^{a}\}_{t\geq 0}$ of $\{V_{ta}\}_{t\geq 0}$, a projection $Q^{a}\in \{V_{ta}, V_{ta}^{*}: \textrm{ ${t\geq 0}$}\}^{'}$ with $(1-Q^{a})\xi^{a}_{t}=\xi^{a}_{t}$ for $t\geq 0$ and $\mu_a \in \mathbb{R}$, such that
\[\Psi_{ta}Exp_{a}(\eta)=e^{\mu_{a} t}e^{\langle\eta|\xi^{a}_{t}\rangle} Exp_{a}(Q^{a}\eta+\xi^{a}_{t})\] for $\eta\in Ker(V_{ta}^{*})$.

For $a\in P$, let $\xi_{a}:=\xi^{a}_{1}$. Proceeding as in the proof of  Thm. \ref{injectivityCAR}, we can prove that 
\begin{enumerate}[(i)]
\item $\{\xi_{a}\}_{a\in P}$ is an additive cocycle for $V$, 
\item there exists a projection $Q \in\{V_{a}, V_{a}^{*}: \textrm{$a\in P$}\}^{'}$ such that $Q|_{Ker(V_{a}^{*})}=Q_{a}$
\item the map $P\ni a\mapsto \mu_{a}$ is a continuous homomorphism. Hence, there exists $\mu \in\mathbb{R}^{d}$ such that $\mu_{a}=\langle \mu|a\rangle$ for $a\in P$.
\end{enumerate}
Therefore, there exists an additive cocycle $\xi=\{\xi_{a}\}_{a\in P}$, a projection $Q\in \{V_{a}, V_{a}^{*}: \textrm{ $a\in P$}\}^{'}$ and a vector $\mu \in\mathbb{R}^{d}$ such that
\begin{equation}\label{projective cocycle}\Psi_{a}Exp_{a}(\eta)=e^{\langle\mu|a\rangle}e^{\langle \eta|\xi_{a}\rangle}Exp_{a}(Q\eta+\xi_{a})\end{equation} for $a\in P$, $\eta\in Ker(V_{a}^{*})$.

Suppose $u=\{u_{a}\}_{a\in P}$ is an exponential unit for $E^{V}$. Let $\eta=\{\eta_{a}\}_{a\in P}\in \mathcal{A}(V)$ be such that $u_{a}=Exp_{a}(\eta_{a})$ for $a\in P$. Since $F$ contains all the units, it follows  that $\Psi_{a}Exp_{a}(\eta_{a})=Exp_{a}(\eta_{a})$ for $a \in P$. In particular, 
\[\Psi_{a}Exp_{a}(0)=Exp_{a}(0)\] for $a\in P$. By Eq. \ref{projective cocycle},
\[e^{\langle\mu|a\rangle}Exp_{a}(\xi_{a})=Exp_{a}(0)\] for $a\in P$. This implies, 
$\xi_{a}=0$ for $a\in P$ and $\mu=0$.

Note that since $Q\in \{V_{a}, V_{a}^{*}: \textrm{ $a\in P$}\}^{'}$, by Remark \ref{disjointness}, $Q$ is a diagonal operator, i.e. for $i\in I$, there exists a projection $Q^{(i)}\in \{V^{(i)}_{a}, V^{(i)*}_{a}: \textrm{ $a\in P$}\}^{'}$ such that $Q|_{\mathcal{H}_{i}}=Q^{(i)}$.
 Fix $i\in I$. By Prop. \ref{additive cocycles that give units}, for any $\eta=\{\eta_{a}\}_{a\in P}\in \mathcal{A}(V^{(i)})$, $Exp(\eta)$ is a unit for $E^{V}$. Hence
\[\Psi_{a}Exp_{a}(\eta_{a})=Exp_{a}(\eta_{a})\] for $\eta=\{\eta_{a}\}_{a\in P}\in \mathcal{A}(V^{(i)})$. This implies that
\begin{equation}\label{Q}Q\eta_{a}=Q^{(i)}\eta_{a}=\eta_{a}\end{equation}for $a\in P$ and $\eta=\{\eta_{a}\}_{a\in P}\in \mathcal{A}(V^{(i)})$.
Note that for $\gamma \in \clk_i$, $\{1_{(0,\langle\lambda_{i}|a\rangle)}\otimes\gamma\}_{a\in P}$ is an additive cocycle for $V^{(i)}$.  Therefore, the set $\{\eta_{a}: \textrm{ $\{\eta_{a}\}_{a\in P}\in \mathcal{A}(V^{(i)})$, $a\in P$}\}$ is total in $\mathcal{H}_{i}$. Now, Eq. \ref{Q} implies that \[Q^{(i)}\eta=\eta\] for $\eta\in \mathcal{H}_{i}$ for $i\in I$, i.e. $Q$ is the identity operator. Hence, 
\[\Psi_{a}Exp_{a}(\xi)=Exp_{a}(\xi)\] for $\xi\in Ker(V_{a}^{*})$ and $a\in P$. 
Since $\{Exp_{a}(\xi): \textrm{$\xi\in Ker(V_{a}^{*})$}\}$ is total in $E^{V}(a)$, it follows that $\Psi_{a}$ is the identity operator for each $a\in P$. Therefore, $F=E^{V}$. Hence the proof.
\end{proof}

\section{Computation of Index and Gauge Group}
 In this section, we compute the index and the gauge group of the product system $E^{(\lambda,k)}$. Arveson's definition of index for a $1$-parameter product system was extended to the multiparameter case in \cite{piyasa} which we first recall. Let $E$ be a product system over $P$. Denote the set of units of $E$ by $\mathcal{U}_{E}$. Assume that $\mathcal{U}_{E}\neq \emptyset$. Fix $a\in Int(P)$. For $u,v\in \mathcal{U}_{E}$,  let $c_{a}(u,v)\in\mathbb{C}$ be such that
\[\langle u_{ta}|v_{ta}\rangle=e^{tc_{a}(u,v)}\] for $t\geq 0$. The function $\mathcal{U}_{E}\times\mathcal{U}_{E} \ni (u,v)\to c_{a}(u,v)\in\mathbb{C}$ is conditionally positive definite and is called the \textbf{covariance function} of $E$ with respect to $a$. 

Let $\mathcal{H}(\mathcal{U}_{E})$ be the Hilbert space obtained from the covariance function using the GNS construction. For the sake of completeness and also to fix notation, we brief the construction of $\mathcal{H}(\mathcal{U}_{E})$. 
Let $\mathbb{C}_{c}(\mathcal{U}_{E})$ denote the  vector space of finitely supported complex valued functions on $\mathcal{U}_{E}$.
Set
\[
\mathbb{C}_0(\mathcal{U}_E):=\{f \in \bbc_c(\mathcal{U}_E): \sum_{u \in \mathcal{U}_E}f(u)=0\}.
\]
Define a semi-definite inner product on $\mathbb{C}_{0}(\mathcal{U}_{E})$ by \[\langle f|g\rangle=\sum_{u,v\in\mathcal{U}}c_a(u,v)f(u)\overline{g(v)}.\]
Let $\mathcal{H}(\mathcal{U}_{E})$ be the Hilbert space obtained by completing the semi-definite inner product  space $\mathbb{C}_0(\mathcal{U}_E)$.

For $u\in \mathcal{U}_{E}$, let $\delta_{u}:\mathcal{U}_E \to \bbc$ be the indicator function $1_{\{u\}}$.
 Clearly, $\{\delta_{u}-\delta_{v}: \textrm{$u,v\in\mathcal{U}_{E}$}\}$ is total in $\mathcal{H}(\mathcal{U}_{E})$. 
The dimension of the space $\mathcal{H}(\mathcal{U}_{E})$ is independent of the choice of $a\in Int(P)$ and is called the \textbf{index} of $E$ denoted $Ind(E)$. The reader is referred to Prop. 2.4 of \cite{piyasa} for a proof of this statement.

 Consider the isometric representation $V=S^{(\lambda,k)}$ for a non-empty countable indexing set $I$, an injective map $\lambda:I\to S(P^{*})$ and a map $k: I\to \bbn_{\infty}$. The isometric representation $V$ shall remain fixed for the rest of this section. Let $\clh:=\displaystyle{\bigoplus_{i\in I}}L^{2}[0,\infty)\otimes \mathcal{K}_{i}$ be the Hilbert space on which $V$ acts, where $\mathcal{K}_{i}$ is a Hilbert space of dimension $k_{i}\in \bbn_{\infty}$ for $i\in I$. For $i\in I$, let $V^{(i)}=S^{(\lambda_{i}, k_{i})}$ which acts  on $\mathcal{H}_{i}=L^{2}[0,\infty)\otimes \clk_{i}$. Denote the set of units of $E^{V}$ by $\mathcal{U}$. Similarly, denote the set of units of $E^{V^{(i)}}$ by $\mathcal{U}_{i}$ for each $i\in I$. Then, $\mathcal{U}_{i}\subset \mathcal{U}$ for $i\in I$.  Let $c(.,.)$ be the covariance function of $E^{V}$ with respect to a fixed $a\in Int(P)$. Let $\mathcal{U}^{\Omega}$, $\mathcal{U}^{\Omega}_{i}$ respectively denote the set of exponential units of $E^{V}$ and $E^{V^{(i)}}$. By Prop. \ref{additive cocycles that give units}, $\mathcal{U}^{\Omega}=\displaystyle{\bigcup_{i\in I}}~\mathcal{U}^{\Omega}_{i}$.

\begin{remark}
\label{modulo exponential}
    
 Let $u\in \mathcal{U}$. Let $\widetilde{u}\in \mathcal{U}^{\Omega}$ be given by $\widetilde{u}_{b}=\frac{u_b}{\langle u_{b}|\Omega_{b}\rangle}$ for $b\in P$,  where $\Omega_{b}$ is the vacuum vector in $E^{V}(b)$.
Fix $a\in Int(P)$. Now, for $v\in\mathcal{U}$, \[\langle u_{ta}|v_{ta}\rangle= \langle u_{ta}|\Omega_{ta}\rangle e^{tc(\widetilde{u}_{ta},v_{ta})}.\] Since the map $[0,\infty)\ni t\to \langle u_{ta}|\Omega_{ta}\rangle\in \bbc^{*}$ is multiplicative, there exists $z_{u}\in \bbc$ such that $\langle u_{ta}|\Omega_{ta}\rangle=e^{tz_{u}}$ for $t\geq 0$. Hence, \[c(u,v)=z_{u}+c(\widetilde{u},v).\]
It is now routine to verify that if $w,z\in \mathcal{U}$, then \[\langle\delta_{u}-\delta_{v}|\delta_{w}-\delta_{z}\rangle=\langle\delta_{\widetilde{u}}-\delta_{\widetilde{v}}|\delta_{w}-\delta_{z}\rangle\] for  every $u,v\in \mathcal{U}$. Thus, for $u,v \in \mathcal{U}$, 
\[\delta_u-\delta_v=\delta_{\widetilde{u}}-\delta_{\widetilde{v}}\] in $\mathcal{H}(\mathcal{U})$. Hence, the set $\{\delta_{u}-\delta_{v}: \textrm{$u,v\in \mathcal{U}^{\Omega}$}\}$ is total in 
$\mathcal{H}(\mathcal{U})$. 
Similarly,  $\{\delta_{u}-\delta_{v}: \textrm{$u,v\in \mathcal{U}^{\Omega}_{i}$}\}$ is total in $\mathcal{H}(\mathcal{U}_{i})$ for $i\in I$.

\end{remark}

\begin{lemma}\label{orthogonal units}If  $u\in \mathcal{U}^{\Omega}_{i}$ and  $v\in \mathcal{U}^{\Omega}_{j}$ for $i\neq j$, then $c(u,v)=0$.
\end{lemma}
\begin{proof} Let $\xi=\{\xi_{b}\}_{b\in P}\in \mathcal{A}(V^{(i)})$ and $\eta=\{\eta_{b}\}_{b\in P}\in \mathcal{A}(V^{(j)})$ be such that $u_{b}=Exp_{b}(\xi_{b})$ and $v_{b}=Exp_{b}(\eta_{b})$ for $b \in P$. Then, for $t\geq 0$, \begin{align*}\langle u_{ta}|v_{ta}\rangle=&e^{\langle\xi_{ta}|\eta_{ta}\rangle} \textrm{ 
   (by Prop. \ref{properties of exponential})}\\
=&1 \textrm{    (since $\langle\xi_{ta}|\eta_{ta}\rangle=0$)}.\end{align*}Therefore, $c(u,v)=0$. Hence the proof.
\end{proof}
Let $a\in Int(P)$ and $i\in I$. Denote the covariance function of $E^{V}$, $E^{V^{(i)}}$ (with respect to $a$) by $c$, $c_{i}$ respectively. Note that if $u,v\in\mathcal{U}_{i}$, $c_{i}(u,v)=c(u,v)$.  For each $i\in I$, the map $\mathcal{H}(\mathcal{U}_{i})\ni \delta_{u}-\delta_{v}\to \delta_{u}-\delta_{v}\in \mathcal{H}(\mathcal{U})$ extends to an isometry from $\mathcal{H}(\mathcal{U}_{i})$ into $\mathcal{H}(\mathcal{U})$. Hence, we may consider $\mathcal{H}(\mathcal{U}_{i})$ as a subspace of $\clh(\mathcal{U})$.

\begin{prop}
\label{index computation}With the foregoing notation, we have the following. 
\begin{enumerate}[(i)]
\item $\mathcal{H}(\mathcal{U}_{i})\perp \mathcal{H}(\mathcal{U}_{j})$ if $i\neq j$,
\item $\mathcal{H}(\mathcal{U})=\displaystyle{\bigoplus_{i\in I}}\mathcal{H}(\mathcal{U}_{i})$, and
\item $Ind(E^{V})=\displaystyle{\sum_{i\in I}}Ind(E^{V^{(i)}})$.
\end{enumerate}
\end{prop}
\begin{proof}Clearly, it suffices to prove $(i)$ and $(ii)$. 
For $i\in I$, let $\mathcal{W}_{i}:=\{\delta_{u}-\delta_{v}: \textrm{$u,v\in \mathcal{U}^{\Omega}_{i}$}\}$. As observed in Remark \ref{modulo exponential},  $\mathcal{W}_{i}$ is total in $\mathcal{H}(\mathcal{U}_{i})$ for $i\in I$. Using Lemma \ref{orthogonal units}, it is routine to verify that \[\langle\delta_{u_{1}}-\delta_{v_{1}}|\delta_{u_{2}}-\delta_{v_{2}}\rangle=0\] whenever $u_{1},v_{1}\in \mathcal{U}_{i}^{\Omega} $ and $u_{2},v_{2}\in \mathcal{U}_{j}^{\Omega}$ and $i\neq j$.  Now $(i)$ follows. 

Consider $\mathcal{W}=\displaystyle{\bigcup_{i\in I}}~\mathcal{W}_{i}\subset \bigoplus_{i\in I}\mathcal{H}(\mathcal{U}_{i})$. Note that if $u\in\mathcal{U}_{i}^{\Omega}$, then $\delta_{u}-\delta_{\Omega}\in \mathcal{W}_i$, where $\Omega$ is the vacuum unit. Let $u,v \in \mathcal{U}^{\Omega}$. Then, there exist $i,j \in I$ such that  $u\in \mathcal{U}_{i}^{\Omega}$ and $v\in \mathcal{U}_{j}^{\Omega}$. Then,  $\delta_{u}-\delta_{v}=(\delta_{u}-\delta_{\Omega})+(\delta_{\Omega}-\delta_{v})$. Consequently, $\delta_{u}-\delta_{v}\in span (\mathcal{W})$. Therefore, $\mathcal{W}$ is total in $\mathcal{H}(\mathcal{U})$, and hence $\displaystyle \mathcal{H}(\mathcal{U})=\bigoplus_{i\in I}\mathcal{H}(\mathcal{U}_{i})$. This proves $(ii)$.
\end{proof}

Let $G$ denote the group of automorphisms of $E^V$, also called the gauge group of $E^V$. 
 Denote the set of normalised units of $E^{V}$ by $\mathcal{U}^{n}$. It follows from Prop. \ref{additive cocycles that give units}  that the normalised units of $E^{V}$ are of the form $\Big\{e^{i\langle\lambda|a\rangle}e^{\frac{-||\xi_{a}||^{2}}{2}}Exp_{a}(\xi_{a})\Big \}_{a \in P}$ for some $\xi=\{\xi_{a}\}\in\mathcal{A}(V^{(i)})$ and $\lambda\in\bbr^{d}$. 
 For each $i\in I$, we denote the unitary group of $\mathcal{K}_{i}$ by $U(k_{i})$.
\vspace{.5 cm}

\noindent \textit{Proof of Thm. \ref{gauge computation}:} 
\begin{enumerate}

\item[(1)] Suppose that $I=\{i\}$ is singleton. Denote $\clk_i$ by $\clk$, $\lambda_i$ by $\lambda$ and $k_i$ by $k$. 
  Let $E^{(\lambda,k)}$ be the product system of the CAR flow associated with $V=S^{(\lambda,k)}$. Since $V$ is a pull back of the one parameter shift semigroup $\{S_t \otimes 1\}$ on $L^2[0,\infty) \otimes \clk$ by the homomorphism $P \ni a \to \langle \lambda|a \rangle \in \bbc$, and since the product systems of CAR and CCR flows associated with the one-parameter shift semigroup $\{S_t \otimes 1\}_{t \geq 0}$ are isomorphic, $E^{(\lambda,k)}$ is isomorphic to  the product system $F^{(\lambda,k)}$ of the CCR flow associated with $V$. In fact,  the map $T=\{T_{a}\}_{a\in P}:E^{(\lambda, k)}\to F^{(\lambda, k)}$ defined by
\[T_{a}( Exp_{a}(\xi))= e(\xi), \] for $\xi\in Ker(V_{a}^{*})$  and $a\in P$ is an isomorphism. Hence, $G$ is isomorphic to the automorphism group of $F^{(\lambda,k)}$.
 Since the automorphism group of a CCR flow acts transitively on the set of normalised units, the conclusion follows. The fact that the gauge group of a CCR flow acts transitively on the set of normalised units can been seen from the explicit description of units and the gauge group obtained in Thm. 5.1 and Thm. 7.3 of \cite{Anbu_Vasanth}. 
 
 Since $E^{(\lambda,k)}$ and $F^{(\lambda,k)}$ are isomorphic, they have the same index. But, by Prop. 2.7 of \cite{piyasa}, \[Ind(F^{(\lambda,k)})=\dim \cla(S^{(\lambda,k)}).\] Note that for every $\eta \in \clk$, $\{1_{(0,\langle \lambda |a \rangle)} \otimes \eta\}_{a \in P}$ is an additive cocycle for $S^{(\lambda,k)}$, and it is not difficult to see that every additive cocycle is of this form. Thus, 
 \[
 Ind(E^{(\lambda,k)})=Ind(F^{(\lambda,k)})=\dim \clk=k.\]

\item[(2)] 
Next assume that $I$ has at least two elements. It is clear from Prop. \ref{index computation} and $(1)$ that 
\[
Ind(E^V)=\sum_{i \in I}k_i.\]
 Let $\Psi=\{\Psi_a\}_{a \in P} \in G$ be given.  
 For $a\in P$, consider the one-parameter product system $E_{a}=\{E(ta)\}_{t\geq 0}$. By Lemma \ref{isostructure}, there exists a unitary $U_{a} \in\{V_{ta}, V_{ta}^{*}: \textrm{ $t\geq 0$}\}^{'}$, an additive cocycle $\xi^{a}=\{\xi^{a}_{t}\}_{t\geq 0}$ and $\mu_{a}\in\mathbb{R}^{d}$ such that
\[\Psi_{ta}Exp_{a}(\eta)=e^{i\mu_{a}t}e^{-\frac{||\xi^{a}_{t}||^{2}}{2}-\langle U_{a}\eta|\xi^{a}_{t}\rangle}Exp_{a}(U_{a}\eta+\xi^{a}_{t})\] for $\eta\in Ker(V_{ta}^{*})$.
Proceeding as in the proof of Theorem \ref{injectivityCAR}, we see that there exists $\xi=\{\xi_{a}\}_{a\in P}\in \mathcal{A}(V)$, 
 a unitary $U\in \{V_{a}, V_{a}^{*}: \textrm{ $a\in P$}\}^{'}$ and $\mu \in\mathbb{R}^{d}$ such that
 \begin{align}\label{automorphism structure}\Psi_{a}Exp_{a}(\eta)=e^{i\langle\mu|a\rangle}e^{-\frac{||\xi_{a}||^{2}}{2}-\langle U\eta|\xi_{a}\rangle}Exp_{a}(U\eta+\xi_{a})\end{align} for $\eta\in Ker(V_{a}^{*})$ and $a\in P$.

  Assume that the additive cocycle $\{\xi_{a}\}_{a\in P}$ is non-zero. Note that $\{\Psi_{a}Exp_{a}(0)\}_{a\in P}$ is a unit of $E^{V}$. Hence,  $\{Exp_{a}(\xi_{a})\}_{a\in P}$ is a unit for $E^{V}$ (the map  $P \ni a \to \langle \xi_a|\eta_a \rangle \in \bbc$ is additive if $\xi,\eta  \in \mathcal{A}(V)$).  By Prop. \ref{additive cocycles that give units}, there exists $i\in I$ such that $\{\xi_{a}\}_{a\in P}\in\mathcal{A}(V^{(i)})$. Let $j\in I$,  $j\neq i$. Let $\{\eta_{a}\}_{a\in P}\in \mathcal{A}(V^{(j)})$ be non-zero. By Prop. \ref{additive cocycles that give units}, $\{Exp_{a}(\eta_{a})\}_{a\in P}$ is a unit for $E^{V}$. Therefore,
 $\{\Psi_{a}Exp_{a}(\eta_{a})\}_{a\in P}$ is a unit for $E^{V}$.  Eq. \ref{automorphism structure} implies that $\{Exp_{a}(U\eta_{a}+\xi_{a})\}_{a\in P}$ is a unit for $E^{V}$. Note that since $U\in \{V_{a},V_{a}^{*}: \textrm{ $a\in P$}\}^{'}$, $U$ is a diagonal operator, i.e. there exists a unitary operator $U_{i}\in \{V^{(i)}_{ta},V^{(i)*}_{ta}: \textrm{$t\geq 0$}\}^{'}$ such that $U|_{\clh_{i}}=U_{i}$ for $i\in I$. Therefore, $U\eta_{a}\in Ker(V^{(j)*}_{a})$ for $a\in P$. Hence, $\{U\eta_{a}+\xi_{a}\}_{a\in P}\not\in \mathcal{A}(V^{(\ell)})$ for any $\ell \in I$, contradicting Prop.\ref{additive cocycles that give units}. As a consequence, $\xi_{a}=0$ for $a\in P$. 
 Therefore, 
\[\Psi_{a}Exp_{a}(\xi)=e^{i\langle\mu |a\rangle}Exp_{a}(U\xi)\] for $\xi\in Ker(V_{a}^{*})$ and $a\in P$. 

Suppose $\mu \in \bbr^{d}$ and $U\in\{V_{a},V_{a}^{*}: \textrm{$a\in P$}\}^{'}$. Then, there exists $\Psi^{(\mu,U)} \in G$ such that  \[\Psi^{(\mu ,U)}_{a}Exp_{a}(\xi)=e^{i\langle\mu |a\rangle}Exp_{a}(U\xi)\] for $\xi\in Ker(V_{a}^{*})$ and $a\in P$.  Note that $\Psi \in G$ because $\Psi_a=e^{i \langle \mu |a \rangle}\Gamma(U)$ for $a \in P$, where $\Gamma(U)$ is the second quantisation map. 

Let $M:=\{V_a,V_a^*:a \in P\}^{'}$ and denote the unitary group of $M$ by $\mathcal{U}(M)$. By Remark \ref{disjointness}, we have $\mathcal{U}(M)=\displaystyle \prod_{i \in I}U(k_i)$. We have shown that the map 
\[
\bbr^d \times \mathcal{U}(M) \ni (\mu ,U) \to \Psi^{(\mu ,U)} \in G
\]
is an isomorphism of groups.   Hence, if $\Psi$ is an automorphism of $E^{V}$, then for $a \in P$, \begin{equation}
      \label{intransitive}
  \Psi_a \Omega_a=\Psi_{a}Exp_{a}(0)=e^{i\langle \mu |a\rangle}Exp_{a}(0)=e^{i\langle \mu |a \rangle}\Omega_a
  \end{equation} for some $\mu \in\bbr^{d}$. Let $\eta=\{\eta_{a}\}_{a\in P}$ be a non-zero additive cocycle such that $Exp(\eta)$ is a  unit.  Then, $u=\{u_{a}\}_{a\in P}$ given by $u_{a}= e^{-\frac{||\eta_{a}||^{2}}{2}}Exp(\eta_{a})$ is a normalised unit, but, by Eq. \ref{intransitive},  $\Psi.\Omega\neq u$ for every $\Psi\in G$. Hence, the action of $G$ on $\mathcal{U}^{n}$ is not transitive. 
 \end{enumerate}
 
 \begin{remark}
Suppose $d \geq 2$. Then, $S(P^*)$ is uncountable. This is because, as $P$ is pointed,  $P^*$ spans $\bbr^d$. It follows from Thm. \ref{gauge computation} and Thm. \ref{main intro} that  there are uncountably many  CAR flows that are type I and for which the action of the gauge group on the set of normalised units is not transitive. 
 
 \end{remark}

\begin{remark}
It is immediate from Thm. \ref{main intro} and Thm. \ref{gauge computation} that if $V$ is an isometric representation with commuting range projections such that $E^V$ is type $I$ and has index one, then $V$ is unitarily equivalent to $S^\lambda$ for some $\lambda \in S(P^*)$, i.e. $E^V$ is `a pullback' of a one parameter CAR flow. The analogous statement for $F^V$, the product system of the CCR flow associated with $V$, is not true (see \cite{piyasa}). 
 \end{remark}
  
\textbf{Question:} Is it possible to construct an isometric representation $V$ of $P$ such that $E^V$ is type $I$, has index one, but $V$ is not unitarily equivalent to $S^\lambda$ for any $\lambda \in S(P^*)$?

\bibliography{reff}

\providecommand{\bysame}{\leavevmode\hbox to3em{\hrulefill}\thinspace}
\providecommand{\MR}{\relax\ifhmode\unskip\space\fi MR }
\providecommand{\MRhref}[2]{%
  \href{http://www.ams.org/mathscinet-getitem?mr=#1}{#2}
}
\providecommand{\href}[2]{#2}
\begin{thebibliography}{10}

\bibitem{anbuCAR}
Anbu Arjunan, \emph{Decomposability of multiparameter {CAR} flows}, Proc.
  Edinb. Math. Soc. (2) \textbf{66} (2023), no.~1, 1--22.

\bibitem{Anbu_Vasanth}
Anbu Arjunan, R.~Srinivasan, and S.~Sundar, \emph{${E}$-semigroups over closed
  convex cones}, J. Operator Theory \textbf{84} (2020), no.~2, 289--322.

\bibitem{Anbu_Sundar}
Anbu Arjunan and S.~Sundar, \emph{{CCR} flows associated to closed convex
  cones}, M\"unster J. of Math \textbf{13} (2020), 115--143.

\bibitem{Arv_Fock}
William Arveson, \emph{Continuous analogues of {F}ock space {I}}, Mem. Amer.
  Math. Soc. \textbf{80} (1989), no.~409.

\bibitem{Arv_Fock3}
\bysame, \emph{Continuous analogues of {F}ock space. {III}. {S}ingular states},
  J. Operator Theory \textbf{22} (1989), no.~1, 165--205.

\bibitem{Arv_Fock2}
\bysame, \emph{Continuous analogues of {F}ock space. {II}. {T}he spectral
  {$C^*$}-algebra}, J. Funct. Anal. \textbf{90} (1990), no.~1, 138--205.

\bibitem{Arv_Fock4}
\bysame, \emph{Continuous analogues of {F}ock space. {IV}. {E}ssential states},
  Acta Math. \textbf{164} (1990), no.~3-4, 265--300.

\bibitem{arveson}
\bysame, \emph{Noncommutative dynamics and {$E$}-semigroups}, Springer
  Monographs in Mathematics, Springer-Verlag, New York, 2003.

\bibitem{Arv}
\bysame, \emph{On the existence of {$E_0$}-semigroups}, Infin. Dimens. Anal.
  Quantum Probab. Relat. Top. \textbf{9} (2006), no.~2, 315--320.

\bibitem{BhatCocycles}
B.~V.~Rajarama Bhat, \emph{Cocycles of {CCR} flows}, Mem. Amer. Math. Soc.
  \textbf{149} (2001), no.~709, x+114. \MR{1804156}

\bibitem{Li_Oberwolfach}
Joachim Cuntz, Siegfried Echterhoff, Xin Li, and Guoliang Yu,
  \emph{{$K$}-theory for group {$C^*$}-algebras and semigroup
  {$C^*$}-algebras}, Oberwolfach Seminars, vol.~47, Birkh\"{a}user/Springer,
  Cham, 2017.

\bibitem{Dahya}
Raj Dahya, \emph{Interpolation and non-dilatable families of
  ${C}_0$-semigroups}, arxiv/math.FA:2307.08565.

\bibitem{KDFK}
Kenneth~R. Davidson, Adam~H. Fuller, and Evgenios T.~A. Kakariadis,
  \emph{Semicrossed products of operator algebras: a survey}, New York J. Math.
  \textbf{24A} (2018), 56--86.

\bibitem{hilgert&neeb}
Joachim Hilgert and Karl-Hermann Neeb, \emph{Wiener-{H}opf operators on ordered
  homogeneous spaces. {I}}, J. Funct. Anal. \textbf{132} (1995), no.~1,
  86--118.

\bibitem{Izumi}
Masaki Izumi, \emph{{$E_0$}-semigroups: around and beyond {A}rveson's work}, J.
  Operator Theory \textbf{68} (2012), no.~2, 335--363.

\bibitem{srinivasanmargett}
Oliver~T. Margetts and R.~Srinivasan, \emph{Invariants for {$E_0$}-semigroups
  on {$II_1$} factors}, Comm. Math. Phys. \textbf{323} (2013), no.~3,
  1155--1184.

\bibitem{MuruganSP}
S.~P. Murugan and S.~Sundar, \emph{On the existence of {$E_0$}-semigroups---the
  multiparameter case}, Infin. Dimens. Anal. Quantum Probab. Relat. Top.
  \textbf{21} (2018), no.~2, 1850007, 20.

\bibitem{Namitha_Sundar}
C.~H. Namitha and S.~Sundar, \emph{Multiparameter {D}ecomposable {P}roduct
  {S}ystems}, arxiv/math.OA: 2112.09414.

\bibitem{Powers_TypeIII}
Robert~T. Powers, \emph{A nonspatial continuous semigroup of
  {$*$}-endomorphisms of {$ B( H)$}}, Publ. Res. Inst. Math. Sci. \textbf{23}
  (1987), no.~6, 1053--1069.

\bibitem{Powers_Index}
\bysame, \emph{An index theory for semigroups of {$*$}-endomorphisms of {$ B(
  H)$} and type {$ II_1$} factors}, Canad. J. Math. \textbf{40} (1988), no.~1,
  86--114.

\bibitem{Powers_Rob}
Robert~T. Powers and Derek~W. Robinson, \emph{An index for continuous
  semigroups of {$*$}-endomorphisms of {$B(H)$}}, J. Funct. Anal. \textbf{84}
  (1989), no.~1, 85--96.

\bibitem{Bhat_Mukherjee}
B.~V. Rajarama~Bhat, J.~Martin Lindsay, and Mithun Mukherjee, \emph{Additive
  units of product systems}, Trans. Amer. Math. Soc. \textbf{370} (2018),
  no.~4, 2605--2637.

\bibitem{Piyasa_Sundar1}
Piyasa Sarkar and S.~Sundar, \emph{Induced isometric representations},
  arxiv/math.OA: 2308.07135.

\bibitem{piyasa}
\bysame, \emph{Examples of multiparameter {CCR} flows with non-trivial index},
  Proc. Edinb. Math. Soc. (2) \textbf{65} (2022), no.~3, 799--832.

\bibitem{Skeide_Shalit}
Orr Shalit and Michael Skeide, \emph{{CP}-semigroups and {D}ilations,
  {S}ub{P}roduct systems and {S}uper{P}roduct systems: {T}he {M}ultiparameter
  {C}ase and {B}eyond}, Dissertationes Mathematicae \textbf{585}, 1--233.

\bibitem{Shalit_2008}
Orr~Moshe Shalit, \emph{{$\rm E_0$}-dilation of strongly commuting {${\rm
  CP}_0$}-semigroups}, J. Funct. Anal. \textbf{255} (2008), no.~1, 46--89.

\bibitem{Shalit}
\bysame, \emph{What type of dynamics arise in {$E_0$}-dilations of commuting
  quantum {M}arkov semigroups?}, Infin. Dimens. Anal. Quantum Probab. Relat.
  Top. \textbf{11} (2008), no.~3, 393--403.

\bibitem{Shalit_dilation}
\bysame, \emph{Dilation theory: a guided tour}, Operator theory, functional
  analysis and applications, Oper. Theory Adv. Appl., vol. 282,
  Birkh\"{a}user/Springer, Cham, [2021] \copyright 2021, pp.~551--623.

\bibitem{Shalit_Skeide_2011}
Orr~Moshe Shalit and Michael Skeide, \emph{Three commuting, unital, completely
  positive maps that have no minimal dilation}, Integral Equations Operator
  Theory \textbf{71} (2011), no.~1, 55--63.

\bibitem{Shalit_Solel}
Orr~Moshe Shalit and Baruch Solel, \emph{Subproduct systems}, Doc. Math.
  \textbf{14} (2009), 801--868.

\bibitem{Skeide}
Michael Skeide, \emph{A simple proof of the fundamental theorem about {A}rveson
  systems}, Infin. Dimens. Anal. Quantum Probab. Relat. Top. \textbf{9} (2006),
  no.~2, 305--314.

\bibitem{Solel}
Baruch Solel, \emph{Representations of product systems over semigroups and
  dilations of commuting {CP} maps}, J. Funct. Anal. \textbf{235} (2006),
  no.~2, 593--618.

\bibitem{srinivasanccr}
R.~Srinivasan, \emph{{CCR} and {CAR} flows over convex cones}, arxiv/math.OA:
  1908.00188.

\bibitem{sundarkms}
S.~Sundar, \emph{On the {KMS} states for the {B}ernoulli shift},
  arxiv/math.OA:2303.12476v2.

\bibitem{Sundar_Ore}
S.~Sundar, \emph{{$C^*$}-algebras associated to topological {O}re semigroups},
  M\"unster J. of Math. \textbf{9} (2016), no.~1, 155--185.

\bibitem{SUNDAR}
\bysame, \emph{Arveson's characterisation of {CCR} flows: the multiparameter
  case}, J. Funct. Anal. \textbf{280} (2021), no.~1, Paper No. 108802, 44.

\bibitem{Sundar_NYJM}
\bysame, \emph{Representations of the weak {W}eyl commutation relation}, New
  York J. Math. \textbf{28} (2022), 1512--1530.

\end{thebibliography}
\bibliographystyle{amsplain} 
\end{document}